\documentclass[oneside,a4paper,reqno]{amsart}
\usepackage{pdfsync, enumerate}
\usepackage{stmaryrd}
\usepackage{mathrsfs}
\usepackage{multicol}
\usepackage{amsmath, amsthm, amscd, amssymb, latexsym, eucal}
\usepackage[all]{xy}
\def\serieslogo@{} \def\@setcopyright{} \makeatother

\usepackage{multienum}

\usepackage[colorinlistoftodos]{todonotes}

\usepackage{hyperref}
\usepackage{color}
\usepackage{cite}

\makeatletter
\renewcommand*\env@matrix[1][c]{\hskip -\arraycolsep
  \let\@ifnextchar\new@ifnextchar
  \array{*\c@MaxMatrixCols #1}}
\makeatother

\usepackage{color}

\numberwithin{equation}{section}
\newtheorem{thm}{Theorem}[section]
\newtheorem*{main-thm}{Main Theorem}
\newtheorem*{Auslander-thm}{Auslander's Theorem}

\newtheorem{cor}[thm]{Corollary}
\newtheorem{lem}[thm]{Lemma}
\newtheorem{prop}[thm]{Proposition}

\newtheorem*{thmA}{Theorem~A}
\newtheorem*{thmB}{Theorem~B}
\theoremstyle{definition}
\newtheorem{defn}[thm]{Definition}
\newtheorem{rem}[thm]{Remark}
\newtheorem{exam}[thm]{Example}

\newcommand{\lxr}{\longrightarrow}

\newcommand{\A}{\mathscr A}
\newcommand{\B}{\mathscr B}
\newcommand{\C}{\mathscr C}

\newcommand{\G}{\mathcal G}

\newcommand{\V}{\mathcal V}

\newcommand{\X}{\mathcal X}

\newcommand{\mr}{\mathsf{r}}
\newcommand{\mq}{\mathsf{q}}
\newcommand{\mi}{\mathsf{i}}
\renewcommand{\mp}{\mathsf{p}}

\newcommand{\ml}{\mathsf{l}}
\newcommand{\me}{\mathsf{e}}

\newcommand{\map}{\mathsf{p}}

 \DeclareMathOperator{\inc}{\mathsf{inc}}
  
\DeclareMathOperator*{\Ker}{\mathsf{Ker}}
 \DeclareMathOperator*{\Image}{\mathsf{Im}}
\DeclareMathOperator*{\Coker}{\mathsf{Coker}}

 \DeclareMathOperator{\pd}{\mathsf{pd}}
\DeclareMathOperator*{\id}{\mathsf{id}}
 
 \DeclareMathOperator{\fpd}{\mathsf{fin.dim}}

  \DeclareMathOperator*{\gld}{\mathsf{gl.dim}}
  \DeclareMathOperator*{\pgld}{\mathsf{pgl.dim}}
\DeclareMathOperator*{\igld}{\mathsf{igl.dim}}
    \DeclareMathOperator*{\findim}{\mathsf{fin.dim}}
    \DeclareMathOperator*{\Findim}{\mathsf{Fin.dim}}
  \DeclareMathOperator*{\maxx}{\mathsf{max}}
 
\DeclareMathOperator*{\Mod}{\mathsf{Mod}-\!}

 \DeclareMathOperator*{\smod}{\mathsf{mod}-\!}

\DeclareMathOperator*{\proj}{\mathsf{proj}}
\DeclareMathOperator*{\Inj}{\mathsf{Inj}}
\DeclareMathOperator*{\Proj}{\mathsf{Proj}}

\DeclareMathOperator*{\add}{\mathsf{add}}

\DeclareMathOperator{\Hom}{\mathsf{Hom}}

\DeclareMathOperator{\Ext}{\mathsf{Ext}}

\newcommand{\evl}{{\operatorname{env}}}
\newcommand{\op}{{\operatorname{op}}}
\newcommand{\rad}{\mathsf{rad}}

\newcommand{\Tip}{{\operatorname{Tip}}}
\newcommand{\Nontip}{{\operatorname{Nontip}}}

\newcommand{\iden}{\operatorname{Id}\nolimits}

\newcommand{\Span}{\operatorname{Span}}

\newsavebox{\proofbox}
\savebox{\proofbox}{\begin{picture}(7,7)%
  \put(0,0){\framebox(7,7){}}\end{picture}}

\usepackage{latexsym}
\usepackage{pstricks}
\usepackage{comment}

\begin{document}


\title[Reduction techniques for the finitistic dimension]
{Reduction techniques for the finitistic dimension}
\author[Green]{Edward L.\ Green}
\address{Department of
Mathematics\\ 
Virginia Tech\\ 
Blacksburg, VA 24061\\
USA} 
\email{green@math.vt.edu}

\author[Psaroudakis]{Chrysostomos Psaroudakis}
\address{Institute of Algebra and Number Theory, University of Stuttgart, Pfaffenwaldring 57, 70569 Stuttgart, Germany}
\email{chrysostomos.psaroudakis@mathematik.uni-stuttgart.de}

\author[Solberg]{\O yvind Solberg}
\address{Department of Mathematical Sciences\\
NTNU\\
N-7491 Trondheim, Norway }
\email{oyvind.solberg@math.ntnu.no}

\date{\today}
 


\begin{abstract}
In this paper we develop new reduction techniques for testing the finiteness of the finitistic dimension of a finite dimensional algebra over a field. Viewing the latter algebra as a quotient of a path algebra, we propose two operations on the quiver of the algebra, namely arrow removal and vertex removal. The former gives rise to cleft extensions and the latter to recollements. These two operations provide us new practical methods to detect algebras of finite finitistic dimension. We illustrate our methods with many examples.
\end{abstract}

\maketitle

\setcounter{tocdepth}{1} \tableofcontents

\section{Introduction}
One of the longstanding open problems in representation theory of
finite dimensional algebras is the {\em Finitistic Dimension
  Conjecture}. Let $\Lambda$ be a finite dimensional algebra over a
field. The finitistic dimension $\findim\Lambda$ of $\Lambda$ is defined as the supremum of the projective
dimension of all finitely generated right modules of finite projective
dimension. The finitistic dimension conjecture asserts that the latter
supremum is finite, i.e. $\findim\Lambda<\infty$. Our aim in this
paper is to present some new reduction techniques for detecting the
finiteness of the finitistic dimension.

The finitistic dimension conjecture has a long and interesting
history. Already in the beginning of the sixties, it became apparent that
the finitistic dimension provides a measure of the complexity of the
module category. In the commutative noetherian case, it has been
proved basically by Auslander and Buchbaum \cite{AuslanderBuchsbaum}
that the finitistic dimension equals the depth of the ring. It was
Bass that emphasized the role of this homological dimension in the
non-commutative setup. For more on the history of the finitistic
dimension conjecture we refer to Zimmermann-Huisgen's paper
\cite{ZimmermannHuisgen}.

The finitistic dimension conjecture is known to be related with other
important problems concerning the homological behaviour and the
structure theory of the module category of a finite dimensional
algebra. In the hierarchy of the homological conjectures in
representation theory, the finitistic dimension conjecture plays a
central role. More precisely, we have the following diagram which
shows that almost all other homological conjectures for finite
dimensional algebras are implied by the finitistic dimension
conjecture $(\mathsf{FDC})\colon$
\[
\xymatrix{
(\mathsf{FDC}) \ar@{=>}[r]^{} \ar@{=>}[d]^{}  & (\mathsf{WTC}) \ar@{=>}[r]^{}  & (\mathsf{GSC})   \\
(\mathsf{NuC}) \ar@{=>}[r]^{}  & (\mathsf{SNC}) \ar@{=>}[r]^{} & (\mathsf{ARC})  \ar@{=>}[r]^{} & (\mathsf{NC})  } 
\]
We write $(\mathsf{SNC})$ for the strong Nakayama conjecture,
$(\mathsf{NC})$ for the Nakayama conjecture, $(\mathsf{ARC})$ for the
Auslander-Reiten conjecture, $(\mathsf{WTC})$ for the Wakamatsu
tilting conjecture, $(\mathsf{NuC})$ for the Nunke condition and
$(\mathsf{GSC})$ for the Gorenstein symmetry conjecture. The above
diagram is not complete, we refer to
\cite{AuslanderReiten, Happel:Gorenstein, Happel, Xi:FinDimConjII} and references therein for more information on
the hierarchy of homological conjectures.

In the middle of the seventies, Fossum--Griffith--Reiten
\cite{FosGriRei} proved for a triangular matrix algebra
$\Lambda =\left(\begin{smallmatrix} R & 0\\ M & S
  \end{smallmatrix}\right)$ that the finitistic dimension of $\Lambda$
is less or equal of the finitistic dimensions of $R$ and $S$ plus one.
Thus, for this particular class of algebras we can test finiteness of
the finitistic dimension by computing the finitistic dimension of the
corner algebras. This result should be considered as the first
reduction technique for the finitistic dimension. Subsequently, but
almost twenty years after, Happel \cite{Happel} showed that if a
finite dimensional algebra $\Lambda$ admits a recollement of bounded
derived categories
$(\mathsf{D}^{\mathsf{b}}(\smod\Lambda''),
\mathsf{D}^{\mathsf{b}}(\smod\Lambda),
\mathsf{D}^{\mathsf{b}}(\smod\Lambda'))$, where $\Lambda'$ and
$\Lambda''$ are finite dimensional algebras, then the finitistic
dimension of $\Lambda$ is finite if and only if the same holds for
$\Lambda'$ and $\Lambda''$. Clearly this is again a reduction
technique for the finitistic dimension. However, it is in general a
difficult problem to decompose the bounded derived category of an
algebra in such a recollement situation. On the other hand, Happel's
technique can be considered as a natural extension of
Fossum-Griffith-Reiten's result, since triangular matrix algebras
(under mild conditions on the bimodule) induces a recollement at the
level of bounded derived categories.

In the beginning of the nineties, Fuller and Saorin
\cite{Ful-Sao:FinDimConjArtRings} introduced the idea of illuminating
simples of projective dimension less or equal to one. In particular,
picking the idempotent $f$ corresponding to such a simple and
considering the idempotent $e=1-f$, they showed that
$\findim{e\Lambda e}<\infty$ implies $\findim{\Lambda}<\infty$. This
is clearly a reduction technique for the finitistic dimension. We
extend this result by putting it in the general context of
recollements of abelian categories. Clearly, this result is the
predecessor of the vertex removal operation. More recently, Xi in a
series of papers \cite{Wang-Xi, Xi:FinDimConjI, Xi:FinDimConjII,
  Xi:FinDimConjIII} introduced various methods for detecting
finiteness of the finitistic dimension. It should be noted that Xi
has connected the finitistic dimension of an algebra of the form
$e\Lambda e$, where $e$ is an idempotent, with other homological
dimensions, for instance, finiteness of the global dimension of
$\Lambda$ (less or equal to four), finiteness of the representation
dimension of $\Lambda/\Lambda e\Lambda$ (less or equal to three) and
several other interesting relations. On the other hand, he considers
pairs of algebras $(B,A)$ where $A$ is an extension of $B$ and the
Jacobson radical $\rad(B)$ is a left ideal in $A$. Then, under
some further conditions he shows that $\findim{A}<\infty$ implies
$\findim{B}<\infty$. Roughly speaking, Xi's philosophy is to control
the finitistic dimension by certain extension of algebras. Clearly,
this machinery is again another reduction technique for testing the
finiteness of the finitistic dimension.

At this point we would like to mention Xi's comment \cite{Wang-Xi} on
the available techniques that we have for the finitistic
dimension. Very briefly, he writes that ``not many practical methods
are available so far to detect algebras of finite finitistic dimension
and it is necessary to develop some methods even for some concrete
examples''. The central point for us in this work is exactly the lack
of practical methods to estimate the finitistic dimension of an algebra.

Let $\Lambda$ be a finite dimensional algebra over an algebraically
closed field $k$. Thus $\Lambda$ is Morita equivalent to an admissible quotient $kQ/I$ of a
path algebra $kQ$ over $k$. Moreover, if $e$ is a trivial path
in $kQ$, then $v_e$ denotes the corresponding vertex in $Q$.  Our
approach on reducing the finitistic dimension is based on two
operations on the quiver of the algebra. It is natural to consider how
the vertices and the arrows contribute to the finitistic
dimension. The idea is to remove those that don't contribute and thus
the finiteness of the finitistic dimension is reduced to a simpler
algebra at least in terms of size. This is clearly a practical method
that can be applied easily to any algebra. 

Let $a$ be an arrow in $Q$, which does not occur in a minimal generating set of $I$, and consider the {\em arrow removal} $\Gamma=\Lambda/\langle
\overline{a}\rangle$. The abstract categorical framework of this operation is the
concept of cleft extension of abelian categories.  Our first main
result provides a ring theoretical characterization of the arrow
removal operation. In addition, it reduces the finiteness of the finitistic dimension of $\Lambda$ to the one of the arrow removal.
The first part of Theorem~A is proved in Proposition~\ref{prop:onearrowchar}. The second part is stated in
Theorem~\ref{thm:summarycleftext} and follows from
Theorem~\ref{thm:cleftextfindim} (a general result on the finitistic
dimension for cleft extensions) and
Proposition~\ref{propremoveiscleft} (a result on the precise
properties of the arrow removal as a cleft extension).

\begin{thmA}\textnormal{(\textsf{Arrow Removal})}
Let $\Lambda = kQ/I$ be an admissible quotient of a path algebra
$kQ$ over a field $k$. Let $a\colon v_e\lxr v_f$ be an arrow in $Q$ and define $\Gamma=\Lambda/\langle
\overline{a}\rangle$ the arrow removal. The following hold.   
\begin{enumerate}[\rm(i)]
\item The arrow $a\colon v_e\lxr v_f$ in $Q$
  does not occur in a set of minimal generators of $I$ in $kQ$ if and
  only if $\Lambda$ is isomorphic to the trivial extension $\Gamma\ltimes
  P$, where $P= \Gamma e\otimes_kf\Gamma$ with $\Hom_\Gamma(e\Gamma,
  f\Gamma)=(0)$. 

\smallskip

\item If the arrow $a$ does not occur in a set of minimal generators of $I$ in $kQ$, then  $\findim\Lambda < \infty$ if and only if $\findim\Gamma < \infty$.

\end{enumerate} 
\end{thmA}

We would like to mention that the arrow removal operation has been
considered in the work of Diracca--Koenig \cite{DK}. Their focus was
removing arrows in a monomial relation and homological reductions
towards the strong no loop conjecture.

Our second operation is the vertex removal. Let $\Lambda$ be a
quotient of a path algebra as above and take $e$ a sum of
vertices. Then the vertices in the quiver of $e\Lambda e$ correspond
to the ones occurring in $e$ and therefore the vertices occurring in
$1-e$ are removed. The transition from $\Lambda$ to $e\Lambda e$ is
what we call {\em vertex removal}. The abstract categorical framework
of this operation is the concept of recollements of abelian
categories. In our second main result we show that removing the
vertices which correspond to simples of finite injective dimension
provides a reduction for the finitistic dimension. This is our second
new practical method for testing the finiteness of the finitistic
dimension. The result is basically proved in
Theorem~\ref{thmigldfinite} in the setting of recollements of abelian
categories.

\begin{thmB}\textnormal{(\textsf{Vertex Removal})}
Let $\Lambda$ be a finite dimensional algebra over an algebraically closed field and $e$ an idempotent element. Then
\[
  \findim\Lambda\leq \findim e\Lambda e+\mathsf{sup}\{\id{_{\Lambda}S}  \ | \ S \ \text{simple} \ \Lambda/\Lambda e\Lambda\text{-module}\}
\] 
\end{thmB}

The above result shows an interesting interplay between the finitistic
dimension and the injective dimension of some simples. Recently, a
similar connection was observed by Rickard \cite{Rickard}. In
particular, he showed that if the injectives over a finite dimensional
algebra generate its unbounded derived category, then the finitistic
dimension conjecture holds.

We have already mentioned several times the practical issue of our
main results. To see this, we advice the reader to look at the last
section where we present several examples where our methods can be
applied without much effort. We have also tested our techniques on
some known examples from the literature.

Our reduction techniques can be applied to any finite dimensional
algebra over an algebraically closed field. Given such an algebra, we
can iterate the reductions (vertex and arrow removal) to obtain a
reduced algebra, see Definition~\ref{defn:reduced}. As a consequence
of our work, \emph{to prove or disprove the finitistic dimension conjecture
it suffices to consider the class of reduced algebras}.

The contents of the paper section by section are as follows. Sections~\ref{sectioncleftextabcat} and~\ref{sectionrecolabcat} are devoted for the abstract categorical framework of the arrow and vertex removal operations on a quiver. In Section~\ref{sectioncleftextabcat} we study cleft extensions of abelian categories. More precisely, we recall and prove several properties of cleft extensions that are used later in Section~\ref{sec:cleft-findim}. We analyze carefully the associated endofuctors that this data carries and we settle the necessary conditions on a cleft extension that the arrow removal operation requires. 

In Section~\ref{sectionrecolabcat} we study recollements of abelian categories. We introduce a relative (injective) homological dimension in a recollement situation and show that it provides interesting homological properties in the abelian categories involved in a recollement, see Proposition~\ref{recollfiniteArelglodimofB}. We also recall the notion of a functor between abelian categories being an eventually homological isomorphism and we characterize when the quotient functor in a recollement is an eventually homological isomorphism, see Proposition~\ref{prop:algebra-ext-iso}. The latter result is used in Section~\ref{sectionvertexremoval}. 

Section~\ref{sec:cleft-findim} is devoted to arrow removal and the finitistic dimension. This section is divided into two subsections. In the first one, we investigate the behaviour of the finitistic dimension of the abelian categories in a cleft extension (under certain conditions), see Theorem~\ref{thm:cleftextfindim}. This is the first key result for showing Theorem~A (ii). In the second subsection, we study arrow removals of quotients of path algebras. We first characterize arrow removals as trivial extensions with projective bimodules admitting special properties, see Corollary~\ref{cor:trivialextwithcond} and Propositions~\ref{prop:Lambda-a-Lambda} and~\ref{prop:onearrowchar}. Then, we show that arrow removals gives rise to cleft extensions satisfying the needed properties for applying Theorem~\ref{thm:cleftextfindim}. This is done in Proposition~\ref{propremoveiscleft}. Finally, we summarise our results on reducing the finitistic dimension by removing arrows in Theorem~\ref{thm:summarycleftext}, where Theorem~A (ii) is a special case.

In Section~\ref{sectionvertexremoval} we investigate the vertex
removal operation with respect to the finitistic dimension. This
section is divided into four subsections. In the first subsection, we
reprove the main result of \cite{GreenMarcos} on reducing the
finitistic dimension via the homological heart using the reduction
techniques of Fossum--Griffith--Reiten and Happel. In the second
subsection, we generalize the result of Fuller--Saorin
\cite[Proposition~2.1]{Ful-Sao:FinDimConjArtRings} (vertex removal,
projective dimension at most one) from the case of artinian rings to
the general context of recollements of abelian categories, see
Theorem~\ref{thmfindimpgld}. We remark that we also provide a lower
bound. The third subsection is about removing vertices of finite
injective dimension. In particular, in Theorem~\ref{thmigldfinite} we
show Theorem~B stated above in the general context of recollements of
abelian cateogories. Note that Theorems~\ref{thmfindimpgld} and
\ref{thmigldfinite} (as well as Theorem~\ref{thm:cleftextfindim} for
the arrow removal) can be applied to the big finitistic dimension as
well. The last subsection is about eventually homological isomorphisms
in recollements of abelian categories and invariance of finiteness of
the finitistic dimension between the middle category and the quotient
category. In particular, we show that $\findim\Lambda<\infty$ if and
only if $\findim{e\Lambda e}<\infty$, where $\Lambda$ is an Artin
algebra and $e$ an idempotent element, provided that the quotient
functor $e(-)\colon \smod\Lambda\lxr \smod{e\Lambda e}$ is an
eventually homological isomorphism, see Theorem~\ref{thmevenfindim}.

The last section, Section~\ref{sectionexamples}, is devoted to examples. We provide a variety of new and known examples where we reduce the finiteness of the finitistic dimension using our techniques of arrow and vertex removal. In particular, we show in Example~\ref{exam:Koszul} that the finitistic dimension of a reduced algebra can be arbitrary large. We also present an example showing that a reduced algebra is not unique.

In an appendix we provide a short introduction to non-commutative Gr\"obner basis for path algebras. We recall some notions and results from the theory of Gr\"obner basis that are used in Section~\ref{sec:cleft-findim}.

\subsection*{Conventions and Notation}
For a ring $R$ we work usually with right $R$-modules and the
corresponding category is denoted by $\Mod{R}$.  The full subcategory
of finitely presented $R$-modules is denoted by $\smod{R}$.  
By a module over an Artin algebra $\Lambda$, we mean a finitely
presented (generated) right $\Lambda$-module. Our abelian categories are assumed to have enough projectives and enough injectives. Given an abelian category $\A$, we denote by $\Proj\A$ (resp. $\Inj\A$) the full subcategory consisting of projective (resp. injective) objects. For an additive functor $F\colon \A \lxr \B$ between
additive categories, we denote by
$\Image F = \{B \in \B \mid B \simeq F(A) \ \text{for some} \ A \in
\A\}$ the {\bf essential image} of $F$ and by
$\Ker F = \{A \in \A \mid F(A) = 0\}$ the {\bf kernel} of $F$. For a path algebra $kQ$, we denote by $v_{e}$ the vertex corresponding to the primitive
idempotent $e$.

\subsection*{Acknowledgments}
The second named author is supported by Deutsche
Forsch\-ungsgemeinschaft (DFG, grant KO $1281/14-1$). The third named
author was partially supported by FRINAT grant number 23130 from the
Norwegian Research Council.  The results of this work were announced
by the second author in the Algebra seminars of Stuttgart and of
Bielefeld. The second author would like to express his gratitude to
the member of the groups for the questions, comments and the useful
discussions.  The software \cite{QPA} was used for initial analysis of
many of the examples in this paper. 

\section{Cleft extensions of abelian categories}
\label{sectioncleftextabcat}

In this section we study cleft extensions of abelian
  categories. This concept was introduced by Beligiannis in
  \cite{Bel:Cleft}, and it generalizes trivial extension of abelian
  categories due to Fossum-Griffith-Reiten \cite{FosGriRei}.  We start
  by recalling and reviewing some of the known results about cleft
  extensions of abelian categories from \cite{Bel:Cleft,
    Bel:Relativecleft}.  Then endofunctors of the two categories
  occurring in a cleft extension are constructed and some properties
  are derived, which are used in Section \ref{sec:cleft-findim} to
  investigate the finitistic dimensions in a cleft extension. 

\subsection{Basis properties}
We first recall the definition of cleft extensions of abelian
categories.

\begin{defn}\textnormal{(\!\!\cite[Definition 2.1]{Bel:Cleft})}
\label{defncleftext}
A {\bf cleft extension} of an abelian category $\B$ is an
abelian category $\A$ together with functors:
\[
\xymatrix@C=0.5cm{
\B \ar[rrr]^{\mi} &&& \A \ar[rrr]^{\me} &&& \B
\ar @/_1.5pc/[lll]_{\ml} } 
\]
henceforth denoted by $(\B,\A, \me, \ml, \mi)$, such that the
following conditions hold:
\begin{enumerate}[\rm(a)]
\item The functor $\me$ is faithful exact.

\item The pair $(\ml,\me)$ is an adjoint pair of functors,
  where we denote the adjunction by
\[\theta_{B,A}\colon \Hom_\A(\ml(B), A) \simeq \Hom_\B(B,\me(A)).\]

\item There is a natural isomorphism
  $\varphi\colon \me\mi\lxr \mathsf{Id}_{\B}$ of functors.
\end{enumerate} 
\end{defn}

Denote the unit $\theta_{B,\ml(B)}(1_{\ml(B)})$ and the counit
$\theta^{-1}_{\me(A),A}(1_{\me(A)})$ of the adjoint pair $(\ml,\me)$
by $\nu\colon 1_\B\lxr \me\ml$ and $\mu\colon \ml\me\lxr 1_\A$,
respectively.  The unit and the counit satisfy the relations 
\begin{equation}\label{eq:unit}
1_{\ml(B)} = \mu_{\ml(B)}\ml(\nu_B)
\end{equation}
and
\begin{equation}\label{eq:counit}
1_{\me(A)} = \me(\mu_A)\nu_{\me(A)}
\end{equation}
for all $B$ in $\B$ and $A$ in $\A$.  From \eqref{eq:counit} the
morphism $\me(\mu_A)$ is an (split) epimorphism.  Using that $\me$ is
faithful exact, we infer that $\mu_A$ is an epimorphism for all $A$ in
$\A$.  Hence we have for every $A$ in $\A$ the following short
exact sequence
\begin{equation}
\label{firstfundamentalsequence}
\xymatrix@C=0.5cm{
0 \ar[rr] && \Ker\mu_A \ar[rr]^{} && \ml\me(A) \ar[rr]^{\mu_A} && A
\ar[rr] && 0  }
\end{equation}

In the next result we collect some basic properties of a cleft
extension. Most of these properties follow from
Definition~\ref{defncleftext} and are discussed in \cite{Bel:Cleft,
  Bel:Relativecleft} but for completeness and the reader's convenience we
provide a proof.

\begin{lem}
\label{basicproperties}
Let $\A$ be a cleft extension of $\B$. Then the following hold.
\begin{enumerate}[\rm(i)]
\item The functor $\me\colon \A\lxr \B$ is essentially surjective. 

\item The functor $\mi\colon \B \lxr \A$ is fully faithful and exact.

\item The functor $\ml\colon \B\lxr \A$ is faithful and preserves
  projective objects.

\item There is a functor $\mq\colon \A\lxr \B$ such that $(\mq,\mi)$
  is an adjoint pair.

\item There is a natural isomorphism $\mq\ml\simeq \mathsf{Id}_{\B}$ of
  functors. 
\end{enumerate}
\begin{proof}
(i) Since $\me\mi \simeq \mathsf{Id}_\B$, it follows immediately that
$\me$ is essentially surjective. 

(ii) $\mi$ \textbf{faithful}: Let $f\colon B\lxr B'$ be in $\B$, and assume that
$\mi(f) = 0$.  Then we have that $\me(\mi(f)) = 0$, and we have a
commutative diagram 
\[\xymatrix{
\me(\mi(B)) \ar[r]^{\me(\mi(f))}\ar[d]_{\varphi_B} & \me(\mi(B'))\ar[d]^{\varphi_{B'}}\\
B \ar[r]^f & B'
}\]
It follows that $f=0$ and $\mi$ is faithful.

$\mi$ \textbf{full}: Let $g\colon \mi(B)\lxr \mi(B')$.  Then
straightforward arguments show that 
\[g = \mi(\varphi_{B'} \me(g)\varphi_{B}^{-1}).\]  
This shows that $\mi$ is full.  

$\mi$ \textbf{exact}: First note the following.  Since
$\me\colon \A\lxr\B$ is faithful, the kernel of $\me$ is consisting
only of the zero object. Let 
\[
\xymatrix@C=0.5cm{
\eta\colon 0 \ar[rr] && B_1 \ar[rr]^{f} && B_2 \ar[rr]^{g} && B_3
\ar[rr] && 0  }
\]
be an exact sequence in $\B$.  Since $\me\mi\simeq \mathsf{Id}_\B$, we
infer that $\me(\mi(\eta))$ is a short exact sequence.  The complex 
\[
\xymatrix@C=0.5cm{
0 \ar[rr] && \mi(B_1) \ar[rr]^{\mi(f)} && \mi(B_2) \ar[rr]^{\mi(g)} && \mi(B_3) \ar[rr] && 0  }  
\] 
gives rise to the exact sequences 
\[
\xymatrix@C=0.5cm{
0 \ar[rr] && \Ker \mi(f) \ar[rr]^{} && \mi(B_1) \ar[rr]^{} && \Image \mi(f) \ar[rr] && 0},
\]
\[
\xymatrix@C=0.5cm{
0 \ar[rr] && \Ker \mi(g) \ar[rr]^{} && \mi(B_2) \ar[rr]^{} && \Image \mi(g) \ar[rr] && 0},
\]
\[
\xymatrix@C=0.5cm{
0 \ar[rr] && \Image \mi(g) \ar[rr]^{} && \mi(B_3) \ar[rr]^{} && \Coker \mi(g) \ar[rr] && 0}.
\]
Applying $\me$ to these sequences and using that $\me$ is faithful
exact and that $\me(\mi(\eta))$ is exact, we conclude that
$\Ker \mi(f) = 0$, $\Image \mi(f) \simeq \Ker \mi(g)$,
$\Image \mi(g) \simeq \mi(B_3)$ and $\Coker \mi(g) = 0$.  Consequently,
the functor $\mi\colon \B\lxr \A$ is exact.

(iii) $\ml$ \textbf{faithful}: Let $f\colon B\lxr B'$ be in $\B$.  For
$X$ in $\B$ let
$\mu'_X\colon \me\ml(X)\lxr X$ be defined by the following 
\[\xymatrix{
\me\ml(\me\mi(X)) \ar[r]^{\me(\mu_{\mi(X)})}
\ar[d]_{\me\ml(\varphi_X)} & \me\mi(X)\ar[d]^{\varphi_X}\\
\me\ml(X) \ar[r]^{\mu'_X} & X
}\]
Then we have the following commutative diagram
\[
\xymatrix{
_\B(B,B) \ar[r]^-\ml\ar[d]^{_\B(B,f)} &
{_\A(\ml(B),\ml(B))}\ar[r]^\simeq \ar[d]^{_\A(\ml(B),\ml(f))} &
{_\B(B,\me\ml(B))}\ar[rr]^{_\B(B,\mu'_B)}
\ar[d]^{_\B(B,\me\ml(f))} & & {_\B(B,B)}\ar[d]^{_\B(B,f)} \\
_\B(B,B')\ar[r]^-\ml & {_\A(\ml(B),\ml(B'))} \ar[r]^\simeq &
{_\B(B,\me\ml(B'))}\ar[rr]^{_\B(B,\mu'_{B'})} & & {_\B(B,B')}
}
\]
Starting with $1_B$ in the upper left corner and tracing this element
to the upper right corner we obtain $1_B$ using \eqref{eq:counit} for
$A = \mi(B)$, hence we get $f$ in the lower right corner.  Starting
with $1_B$ in the upper left corner and tracing it around the first
square we obtain $\ml(f)$.  Using that $\ml(f)$ is mapped to $f$ on
the lower row, it follows that the functor $\ml$ is faithful.  

$\ml$ \textbf{preserves projectives}: Since we have the adjoint pair
$(\ml,\me)$ and the functor $\me$ is exact, it follows that the
functor $\ml\colon \B\lxr\A$ preserves projective objects.

(iv) We show that there is a functor $\mq\colon \A\lxr \B$ such that
$(\mq, \mi)$ is an adjoint pair. Consider first the maps $\me \mu_A$
and $\me \mu_{\mi\me(A)}$ in the following not necessarily
commutative diagram:
\[
\xymatrix{
\me\ml\me\mi\me(A) \ar[r]^{\me\mu_{\mi\me(A)}}
\ar[d]^{\me\ml\varphi_{\me(A)}}_{\simeq} & \me\mi\me(A)
\ar[d]^{\varphi_{\me(A)}}_{\simeq}  \\ 
\me\ml\me(A) \ar[r]^{\me\mu_A} & \me(A) }
\]
For simplicity, we identify the vertical isomorphisms and we consider
the following exact sequence:
\begin{equation}
\label{exactseqdefq}
\xymatrix{
\me\ml\me(A) \ar[rr]^{\me\mu_A-\me\mu_{\mi\me(A)}} && \me(A)
\ar[r]^{\kappa_A \ \ \ \ \ \ \ \ \ \ \ \ } &
\Coker{(\me\mu_A-\me\mu_{\mi\me(A)})} \ar[r]^{} & 0} 
\end{equation}
Then there is a functor $\mq\colon \A\lxr \B$ defined on objects by
the assignment
$A\mapsto \mq(A):=\Coker{(\me\mu_A-\me\mu_{\mi\me(A)})}$. Given a
morphism $f\colon A\lxr A'$ in $\A$, then $\mq(f)$ is the induced
morphism in $\B$ between the cokernels
$\mq(f)\colon \mq(A)\lxr \mq(A')$.  For $A=\mi(B)$ in
$(\ref{exactseqdefq})$, it follows that the map
$\me\mu_{\mi(B)}-\me\mu_{\mi\me(\mi(B))}=0$, since
  $\varphi_{\me\mi(B)}  = \me\mi(\varphi_B)$. Therefore we get that
the map $\kappa_{\mi(B)}\colon B\lxr \mq\mi(B)$ is an isomorphism. We
define a natural morphism:
\[
F_{A,B}\colon \Hom_{\A}(A,\mi(B))\lxr \Hom_{\B}(\mq(A), B), \,
f\mapsto \kappa_{\mi(B)}^{-1}  \mq(f)\colon q(A)\lxr B 
\]
Since $\kappa_{\mi(B)}\me(f) = \mq(f)\kappa_A$ and
  $\me$ is faithful,
the map $F_{A,B}$ is injective. To show that $F_{A,B}$ is 
also an epimorphism we need some more work. For every $A$ in $\A$ we
have the short exact sequence $(\ref{firstfundamentalsequence})$. In
particular, using that $\mu_A\colon \ml\me(A)\lxr A$ is surjective for
all objects $A$ in $\A$ we get the following exact sequence:
\[
\xymatrix{
\ml\me\ml\me(A) \ar[rr]^{\ml\me\mu_A-\mu_{\ml\me(A)}} && \ml\me(A) \ar[r]^{\mu_A} & A\ar[r]^{} & 0}
\]
It is shown in \cite[Proposition~2.3]{Bel:Cleft} that the next composition is zero:
\[
\xymatrix{
\ml\me\ml\me(A) \ar[rr]^{\ml\me\mu_A-\mu_{\ml\me(A)}} && \ml\me(A) \ar[r]^{\ml\varphi_{\me(A)}^{-1}} & \ml\me\mi\me(A)\ar[r]^{\mu_{\mi\me(A)}} & \mi\me(A) \ar[r]^{\mi\kappa_A \ } & \mi\mq(A) }
\]
This implies that there is map $\lambda_A\colon A\lxr \mi\mq(A)$ such
that
\begin{equation}
\label{equationcokernel}
\lambda_A  \mu_A = \mi\kappa_A  \mu_{\mi\me(A)} 
\ml\varphi_{\me(A)}^{-1} \tag{$*$}
\end{equation}
Consider now a morphism $g\colon \mq(A)\lxr
B$. Then the composition map
$\mi(g)  \lambda_A\colon A\lxr \mi(B)$ is such that
$F_{A,B}(\mi(g)  \lambda_A)=g$. To see this there is a series of computations that the reader needs to verify. First, we compute that $F_{A,B}(\mi(g)\lambda_A)=\kappa_{\mi(B)}^{-1} \mq\mi(g) \mq(\lambda_A)$ and we have to show that the latter morphism is $g$. The first observation is that $\kappa_{\mi(B)}^{-1}  \mq\mi(g) = \me\mi(g)  \kappa_{\mi\mq(A)}^{-1}$. Using the natural isomorphism $\varphi_B$ and $\varphi_{\mq(A)}$ we obtain that the desired composition $\kappa_{\mi(B)}^{-1} \mq\mi(g) \mq(\lambda_A)$ is $g  \kappa_{\mi\mq(A)}^{-1}  \mq(\lambda_A)$. Using the relation \eqref{equationcokernel}, it follows that $\me \lambda = \kappa$. Since $\kappa_{\mi\mq(A)}  \me\lambda_A=\mq(\lambda_A)  \kappa_A$ using the identification $\me\lambda = \kappa$ we get that $\mq(\lambda_A)  \kappa_A = \kappa_{\mi\mq(A)}  \kappa_A$. Since the map $\kappa_A$ is an epimorphism, it follows that $\mq(\lambda_A)=\kappa_{\mi\mq(A)}$. Thus the desired composition  $g  \kappa_{\mi\mq(A)}^{-1}  \mq(\lambda_A)$ gives the map $g$. This shows that $F_{A,B}$ is
surjective. The details are left to the reader, see also the proof of \cite[Proposition~2.3]{Bel:Cleft}. We infer that $(\mq,\mi)$ is an adjoint pair.

(v) From the adjoint pairs $(\ml, \me)$, $(\mq, \mi)$ and since we have the natural isomorphism $\me\mi\simeq\mathsf{Id}_{\B}$, it follows that there is a natural isomorphism $\mq\ml\simeq \mathsf{Id}_{\B}$ of
functors.
\end{proof}
\end{lem}

\subsection{Endofunctors}
We saw in \eqref{firstfundamentalsequence} that there is a short exact
sequence 
\[\xymatrix@C=0.5cm{
0 \ar[rr] && \Ker\mu_A \ar[rr]^{} && \ml\me(A) \ar[rr]^{\mu_A} && A
\ar[rr] && 0  }\]
for all $A$ in $\A$.  The assignment $A\mapsto \Ker\mu_A$ defines an
endofunctor $G\colon \A\lxr \A$. Consider now an object $B$ in
$\B$. Then we have the short exact sequence in $\A$: 
\[
\xymatrix@C=0.5cm{
0 \ar[rr] && G(\mi(B)) \ar[rr]^{} && \ml\me(\mi(B)) \ar[rr]^{\mu_{\mi(B)}} && \mi(B)
\ar[rr] && 0  }
\]
and applying the exact functor $\me\colon \A\lxr \B$ we obtain the
exact sequence
\[
\xymatrix@C=0.5cm{
0 \ar[rr] && \me(G(\mi(B))) \ar[rr]^{} && \me(\ml\me(\mi(B))) \ar[rr]^{\me(\mu_{\mi(B)})} && \me(\mi(B))
\ar[rr] && 0  }
\]
We denote by $F(B)$ the object $\me(G(\mi(B)))$. Then the
assignment $B\mapsto F(B)$ defines an endofunctor $F\colon \B\lxr
\B$. Viewing the natural isomorphism $\me\mi(B)\simeq B$ as an
identification, we obtain the exact sequence:
\begin{equation}
\label{splitsequencegivesF}
\xymatrix@C=0.5cm{
0 \ar[rr] && F(B) \ar[rr]^{} && \me(\ml(B)) \ar[rr]^{ \ \ \me(\mu_{\mi(B)})} && B
\ar[rr] && 0  }
\end{equation}
The following lemma is an immediate consequence of \eqref{eq:counit}. 
\begin{lem}
\label{lemsplitexactseq}
Let $(\B,\A,\me, \ml, \mi)$ be a cleft extension of abelian
categories. Then the exact sequence $(\ref{splitsequencegivesF})$
splits.
\end{lem}

We end this section by discussing cleft extensions with special
  properties.  In the application discussed in Section
  \ref{sec:cleft-findim} the square of $F$ is zero, and the functor
  $\ml$ is exact and the functor $\me$ preserves projectives.  The
  remaining results concern cleft extensions having some of these
  properties or generalizations thereof. 
  
\begin{lem}
\label{lemnilpotentfunctors}
Let $(\B,\A,\me, \ml, \mi)$ be a cleft extension of abelian categories.
The following statements hold.
\begin{enumerate}[\rm(i)]
\item For any $n\geq 1$, there is a natural isomorphism $\me G^n\simeq F^n\me$.

\item Let $n\geq 1$. Then $F^n=0$ if and only if $G^n=0$.
\end{enumerate}
\begin{proof}
  (i) Let $A$ be an object in $\A$. We first show that
  $\me(G(A))\simeq F(\me(A))$. From \eqref{firstfundamentalsequence}
  and \eqref{splitsequencegivesF} we obtain the following two short
  exact sequences:
\[
\xymatrix@C=0.5cm{
0 \ar[rr] && \me(G(A)) \ar[rr]^{} && \me\ml\me(A) \ar[rr]^{\me(\mu_A)} && \me(A) \ar[rr] && 0  }
\]
and 
\[
\xymatrix@C=0.5cm{
0 \ar[rr] && F(\me(A)) \ar[rr]^{} && \me\ml\me(\mi\me(A)) \ar[rr]^{ \ \ \me(\mu_{\mi\me(A)})} && \me\mi\me(A)
\ar[rr] && 0  }
\]
where by definition $F(\me(A))=\me G(\mi\me(A))$. From
Lemma~\ref{lemsplitexactseq} the above two exact sequences split. Then
using the natural isomorphism
$\varphi\colon \me\mi\lxr \mathsf{Id}_{\B}$ we obtain the following
exact commutative diagram:
\[
\xymatrix@C=0.5cm{
0 \ar[rr] && \me\mi\me(A) \ar[d]_{\simeq}^{\varphi_{\me(A)}} \ar[rr]^{\nu_{\me\mi\me(A)} \ } && \me\ml(\me\mi\me(A)) \ar[d]_{\simeq}^{\me\ml\varphi_{\me(A)}} \ar[rr]^{} && F(\me(A)) \ar[rr] \ar@{-->}[d] && 0 \\
0 \ar[rr] && \me(A) \ar[rr]^{\nu_{\me(A)} \ } && \me\ml\me(A) \ar[rr]^{}  && \me(G(A)) \ar[rr]  && 0 }
\]
From the Snake Lemma it follows that $\me(G(A))\simeq F(\me(A))$.  By
induction it follows that $\me G^n(A) \simeq F^n\me(A)$ for all $n\geq
1$.  Hence $\me G^n \simeq F^n\me$ for all $n\geq 1$. 

(ii) Let $n\geq 1$.  If $F^n=0$, then $\me G^n = 0$ by (i).  Since
$\me$ is faithful, $G^n = 0$.  If $G^n = 0$, then $F^n\me = 0$ by
(i).  Since $\me$ is essentially surjective by Lemma \ref{basicproperties} (i),
we infer that $F^n = 0$. This completes the proof. 
\end{proof}
\end{lem}

\begin{rem}
There is a dual notion of a cleft extension for abelian categories. More precisely, a cleft coextension of an abelian category $\B$ is an abelian category $\A$ together with functors:
\[
\xymatrix@C=0.5cm{
\B \ar[rrr]^{\mi} &&& \A \ar[rrr]^{\me} &&& \B
\ar @/^1.5pc/[lll]_{\mr} } 
\]
such that the functor $\me$ is faithful exact, $(\me,\mr)$ is an adjoint pair of functors, and there is a natural isomorphism $\me\mi\simeq \mathsf{Id}_{\B}$ of functors. In this case, we can derive as in Lemma~\ref{basicproperties} similar properties for a cleft coextension. It is very interesting when a cleft extension of abelian categories $(\B,\A,\me, \ml, \mi)$ is also a cleft coextension. Indeed, this holds if and only if the endofunctor $F$ appearing in the split exact sequence $(\ref{splitsequencegivesF})$ has a right adjoint. In this case, we have the following diagram of functors:
\begin{equation}
\label{diagramdualcleft}
\xymatrix@C=0.5cm{
\B  \ar@(ul,dl)_{F} \ar[rrr]^{\mi} &&& \A  \ar@(ul,ur)^{G} \ar[rrr]^{\me} \ar @/_1.5pc/[lll]_{\mq}  \ar
 @/^1.5pc/[lll]_{\map} &&& \B  \ar@(ur,dr)^{F}
\ar @/_1.5pc/[lll]_{\ml} \ar
 @/^1.5pc/[lll]_{\mr} }
\end{equation}
All these concepts are due to Beligiannis \cite{Bel:Cleft,
  Bel:Relativecleft}, in particular, see \cite[subsection 2.4]{Bel:Relativecleft} for a thorough discussion of cleft coextensions of abelian categories. We remark that diagram~\ref{diagramdualcleft} will be studied further in Section~\ref{sec:trivext}.
\end{rem}

From now on we make the following assumption on a cleft extension $(\B,\A,\me, \ml, \mi)$ of abelian categories:
\begin{equation}
\label{assumptionscleftext}
\text{The functor} \ \ml \ \text{is exact and the functor} \ \me \ \text{preserves projectives.}
\end{equation}
Note that if a cleft extension $(\B,\A,\me, \ml, \mi)$ is also a cleft coextension, i.e. we have diagram $(\ref{diagramdualcleft})$, then $\mr$ being exact implies that $\me$ preserves projectives. 

We continue with the following useful results.

\begin{lem}
\label{corpropofF}
Let $(\B,\A, \me, \ml, \mi)$ be a cleft extension of abelian categories such that condition $(\ref{assumptionscleftext})$ holds. Then the functor $F$ is exact and preserves projective objects.
\begin{proof}
From Lemma~\ref{lemsplitexactseq} it follows that there is an
  isomorphism of functors $\me\ml \simeq F \oplus 1_\B$. Since the functors
$\me$ and $\ml$ are exact, we infer that $F$ is exact. 

By Lemma \ref{basicproperties} (iii) the functor $\ml$ preserves
projective objects, and by assumption the functor $\me$ has the same
property.  Then the isomorphism $\me\ml \simeq F \oplus 1_\B$ show
that $F$ preserves projective objects.
\end{proof}
\end{lem}

\begin{lem}\label{lem:FGequivalence}
Let $(\B,\A, \me, \ml, \mi)$ be a cleft extension of abelian categories such that condition $(\ref{assumptionscleftext})$ holds. Moreover, assume that $F^n=0$ for some $n\geq 2$. Let $A$ be an object in $\A$. Then the following statements are equivalent:
\begin{enumerate}[\rm(i)]
\item $G^{n-1}(A)$ lies in $\Proj\A$. 

\item $F^{n-1}\me(A)$ lies in $\Proj\B$. 
\end{enumerate}
\begin{proof}
  (i)$\Longrightarrow$(ii) Assume that $G^{n-1}(A)$ is projective in
  $\A$.  By Lemma \ref{basicproperties} (vii) the functor $\me$
  preserves projective objects, hence $\me G^{n-1}(A)$ is projective
  in $\B$.  By Lemma \ref{lemnilpotentfunctors} (i) the objects
  $\me G^{n-1}(A)$ and $F^{n-1}\me(A)$ are isomorphic, and therefore
  $F^{n-1}\me(A)$ is projective in $\B$.

  (ii)$\Longrightarrow$(i) Asume that $F^{n-1}\me(A)$ is projective in
  $\B$. Then $\ml F^{n-1}\me(A)$ lies in $\Proj\A$ by Lemma
  \ref{basicproperties} (iii), and therefore from
  Lemma~\ref{lemnilpotentfunctors} the object $\ml\me G^{n-1}(A)$ is
  projective.  Moreover, Lemma~\ref{lemnilpotentfunctors} (ii) yields
  that $G^n=0$ by our assumption.  Then, if we consider the exact
  sequence $(\ref{firstfundamentalsequence})$ for the object
  $G^{n-1}(A)$, it follows that the map
  $\mu_{G^{n-1}(A)}\colon \ml\me G^{n-1}(A)\lxr G^{n-1}(A)$ is an
  isomorphism. Hence, the object $G^{n-1}(A)$ is projective.
\end{proof}
\end{lem}

\section{Recollements of abelian categories}
\label{sectionrecolabcat}

We start this section by recalling the definition of a recollement
situation in the context of abelian categories, see for instance
\cite{Psaroud:survey} and references therein, we fix notation and
recall some well known properties of recollements which are used later
in the paper. We also introduce a relative homological dimension in a recollement $(\A,\B,\C)$ which is relevant for the inclusion functor $\mi\colon \A\lxr \B$ to be a homological embedding and provides other properties of the recollement. Finally, we recall when a functor between abelian categories is an eventually homological isomorphism and characterize this property for the quotient functor in a recollement.

We begin by recalling the definition of a recollement of abelian categories.

\begin{defn}
\label{defnrec}
A {\bf recollement situation} between 
abelian categories $\A,\B$ and $\C$ is a diagram
\[
\xymatrix@C=0.5cm{
\A \ar[rrr]^{\mi} &&& \B \ar[rrr]^{\me}  \ar @/_1.5pc/[lll]_{\mathsf{\mq}}  \ar
 @/^1.5pc/[lll]^{\mp} &&& \C
\ar @/_1.5pc/[lll]_{\ml} \ar
 @/^1.5pc/[lll]^{\mr}
 } 
\]
henceforth denoted by $(\A,\B,\C)$, satisfying the following
conditions:
\begin{enumerate}[\rm(a)]
\item $(\ml,\me)$ and $(\me,\mr)$ are adjoint pairs.
\item $(\mq,\mi)$ and $(\mi,\mp)$ are adjoint pairs.
\item The functors $\mi$, $\ml$, and $\mr$ are fully faithful.
\item $\Image\mi = \Ker \me$.
\end{enumerate}
\end{defn}

We collect some basic properties of a recollement of abelian categories. They can be derived easily from
Definition~\ref{defnrec}, for more details see \cite{Psaroud:survey}.

\begin{prop}
\label{properties}
Let $(\A,\B,\C)$ be a recollement of abelian categories.  Then the
following hold.
\begin{enumerate}[\rm(i)]
\item The functors $\mi\colon \A\lxr \B$ and $\me\colon \B\lxr \C$ are
exact.
\item The compositions $\me\mi$, $\mq\ml$ and $\mp\mr$ are zero.
\item The functor $\me\colon \B\lxr \C$ is essentially surjective.
\item The units of the adjoint pairs $(\mi,\mp)$ and $(\ml,\me)$ and the
counits of the adjoint pairs $(\mq,\mi)$ and $(\me,\mr)$ are
isomorphisms:
\[
\iden_\A \xrightarrow{\simeq} \mp\mi \qquad
\iden_\C \xrightarrow{\simeq} \me\ml \qquad
\mq\mi \xrightarrow{\simeq} \iden_\A  \qquad
\me\mr \xrightarrow{\simeq} \iden_\C
\]
\item The functors $\ml\colon \C\lxr \B$ and $\mq\colon \B\lxr \A$
preserve projective objects and the functors $\mr\colon \C\lxr \B$ and
$\mp\colon \B\lxr \A$ preserve injective objects.
\item The functor $\mi\colon \A\lxr \B$ induces an equivalence between
$\A$ and the Serre subcategory $\Ker\me = \Image\mi$ of $\B$.  Moreover,
$\A$ is a localizing and colocalizing subcategory of $\B$ and there is
an equivalence of categories $\B/\A \simeq \C$.
\item For every $B$ in $\B$ there are objects $A$ and $A'$ in $\A$ such that
the units and counits of the adjunctions induce the following exact
sequences:
\begin{equation}\label{Bsequence}
 \xymatrix{
  0 \ar[r]^{ } & \mi(A) \ar[r]^{ } & \ml\me(B) \ar[r]^{ } &
  B  \ar[r]^{ } & \mi\mq(B) \ar[r] & 0  }
\end{equation}
and 
\[
\xymatrix{
  0 \ar[r]^{ } & \mi\mp(B) \ar[r]^{ } & B \ar[r]^{ } &
  \mr\me(B)  \ar[r] & \mi(A') \ar[r] & 0  }
\]  
\end{enumerate}
\end{prop}

A well-known example of a recollement of abelian categories is
  induced from a ring and an idempotent as we explain next.
  
\begin{exam}
\label{examrecolmodcat}
Let $R$ be a ring with an idempotent element $e \in R$. Then the diagram 
\[
\xymatrix@C=0.5cm{
\Mod{R/ReR} \ar[rrr]^{\inc} &&& \Mod{R} \ar[rrr]^{e(-)} \ar
@/_1.5pc/[lll]_{-\otimes_R R/ReR}  \ar
 @/^1.5pc/[lll]^{\Hom_R(R/ReR,-)} &&& \Mod{eRe}
\ar @/_1.5pc/[lll]_{-\otimes_{eRe} eR} \ar
 @/^1.5pc/[lll]^{\Hom_{eRe}(Re,-)}
 } 
\]
is a recollement of abelian categories. It should be noted that this
recollement is the universal example of a recollement situation with
terms categories of modules. Indeed, from \cite{PsaroudVitoria} we
know that any recollement of module categories is equivalent, in an
appropriate sense, to one induced by an idempotent element as above.
\end{exam}
There are even further functors that are naturally associated to a
recollement of abelian categories $(\A,\B,\C)$.  One such functor is obtained from
the exact sequence \eqref{Bsequence}: We let $H\colon \B\lxr \B$ be
the endofunctor defined by the short exact sequence
\begin{equation}
\label{endofunctorexactseq}
\xymatrix{
  0 \ar[r]^{ } & H(B) \ar[r]^{ } &
  B  \ar[r]^{ } & \mi\mq(B) \ar[r] & 0  }
\end{equation}
on all objects $B$ of $\B$. The endofunctor $H$ is an idempotent radical subfunctor of the identity functor $\iden_{\B}$, see \cite[Proposition 3.5]{Psaroud}. This endofunctor is useful in connection with the next relative dimensions, see Proposition~\ref{recollfiniteArelglodimofB} below.

Let $(\A,\B,\C)$ be a recollement of abelian categories. In
\cite{Psaroud} the notion of $\A$-relative global
dimension\footnote{Note that in \cite{Psaroud} the notation was
${\pgld}_{\A}{\B}$.} of $\B$
was defined as follows:
\[
{\pgld}_{\B}{\A}:=\sup \{\pd_{\B}\mi(A) \ | \ A\in \A\}
\]
We call this the $\A$-{\bf relative projective global dimension} of
$\B$. Dually, we define the
$\A$-{\bf relative injective global dimension} of $\B$ by
\[
{\igld}_{\B}{\A}:=\sup \{{\id}_{\B}\mi(A) \ | \ A\in \A\}.
\]

Our aim is to explore the homological behaviour of a recollement
$(\A,\B,\C)$ under the finiteness of $\pgld_{\B}{\A}$ or
$\igld_{\B}{\A}$.

From Proposition~\ref{properties} we know that the functor
$\mr\colon \C\lxr \B$ preserves injective objects. Under the finiteness
of $\igld_{\B}{\A}$, we show that the functor $\mr$ preserves objects of
finite injective dimension even though it is not an exact
functor in general. Similar considerations hold for the functor $\ml$.

\begin{lem}
\label{lempigldfinite}
Let $(\A,\B,\C)$ be a recollement of abelian categories.
\begin{enumerate}[\rm(i)]
\item If $\igld_{\B}\A<\infty$, then the functor $\mr\colon \C\lxr \B$
  preserves objects of finite injective dimension.

\item If $\pgld_{\B}\A<\infty$, then the functor $\ml\colon \C\lxr \B$
  preserves objects of finite projective dimension.
\end{enumerate}
\begin{proof}
We only prove (i) as (ii) is shown by similar arguments. 

For any object $C$ in $\C$ we show that the following formula holds:
\begin{equation}
\label{formulainjectdim}
{\id}_{\B}\mr(C) \leq {\id}_{\C}C+{\igld}_{\B}\A+1
\end{equation}
Then it follows immediately that the functor $\mr$ preserves objects
of finite injective dimension. Note that the dual formula is proved in
\cite[Lemma 4.3]{Psaroud}, so (ii) follows as well. For readers
convenience we prove formula $(\ref{formulainjectdim})$.

Assume that $\igld_{\B}\A=n<\infty$. If the object $C$ is injective,
then $\mr(C)$ is injective and therefore the relation
$(\ref{formulainjectdim})$ holds. Suppose that the object $C$ has
injective dimension one. This means that there is an exact sequence
\[
 \xymatrix{
  0 \ar[r]^{} & C \ar[r]^{a^0} & I^0 \ar[r]^{a^1} &
  I^1  \ar[r]^{ } & 0  }
\] 
with $I^0$ and $I^1$ in $\Inj{\C}$. Applying the left exact functor $\mr$ we obtain the following exact sequence
\[
 \xymatrix{
  0 \ar[r]^{} & \mr(C) \ar[r]^{\mr(a^0)} & \mr(I^0) \ar[r]^{\mr(a^1)} &
  \mr(I^1) \ar[r] & \mathsf{R}^1\mr(C)  \ar[r]^{ } & 0  }
\] 
where $\mathsf{R}^1\mr(C)$ is the first right derived functor of
$\mr$. If we apply the exact functor $\me\colon \B\lxr \C$ and by
Proposition~\ref{properties} (iv), we derive that the object
$\mathsf{R}^1\mr(C)$ is annihilated by $\me$. This implies that
$\mathsf{R}^1\mr(C)$ lies in $\mi(\A)$ which by our assumption satisfies $\id_{\B}\mathsf{R}^1\mr(C)\leq n$. Thus, we have
$\id_{\B}\Coker\mr(a^0)\leq n+1$ and therefore we conclude that
$\id_{\B}\mr(C)\leq 1+n+1$. Continuing inductively on the length of
the injective resolution of $C$ we get formula
$(\ref{formulainjectdim})$ and this completes the proof.
\end{proof}
\end{lem}

For an injective coresolution $0\lxr B\lxr I^0\lxr I^1\lxr \cdots$ of
$B$ in $\B$, the image of the morphism $I^{m-1}\lxr I^m$ is an $m$-th
cosyzygy of $B$ and is denoted by $\Sigma^m(B)$. The full subcategory
of $\B$ consisting of the $m$-th cosyzygy objects is denoted by
$\Sigma^m(\B)$.

The finiteness of the $\A$-relative projective global dimension of
$\B$ implies also that $\sup\{\pd_{\B}\mi(P) \ | \ P\in \Proj\A\}$ is
finite. In the following result we characterize the finiteness of the
above number, compare this with \cite[Remark 4.6, Proposition
4.15]{Psaroud}. Recall from \cite{Psaroud} that an exact functor
$\mi\colon \A\lxr \B$ between abelian categories is a {\bf homological
  embedding} if the map $\mi^n_{X,Y}\colon\Ext^n_{\A}(X,Y)\lxr
\Ext^n_{\B}(\mi(X),\mi(Y))$ is an isomorphism for every pair of objects $X,Y$ in $\A$ and for all $n\geq 0$. To this end
we need to introduce the following notation.   

Let $\X \subseteq \A$ be a full subcategory. For integers
$0 \leq i \leq k$ we denote by $\X^{{_i\bot_k}}$ the full
subcategory of $\A$ which is defined by
\[
\X^{{_i\bot_k}}=\{A\in \A \ | \ \Ext_{\A}^n(\X, A)=0, \ \forall \
    i\leq n\leq k \}.
\] 
We also denote by $\X^{{_1\bot_\infty}}$ the full subcategory of
$\A$ defined by
\[
\X^{{_1\bot_\infty}} = \{A \in \A \ | \ \Ext_{\A}^n(\X, A)=0, \
  \forall n\geq 1\}.
\] 
Similarly we define the full subcategories $^{{_i\bot_k}}\X$ and
$^{{_1\bot_{\infty}}}\X$.

\begin{prop}\label{recollfiniteArelglodimofB}
Let $(\A, \B, \C)$ be a recollement of abelian categories. Assume
that $\B$ has enough projectives and injectives. The following are
equivalent:
\begin{enumerate}[\rm(i)]
\item $\sup\{\pd_{\B}\mi(P) \ | \ P\in \Proj\A\}\leq m$.

\item $\Sigma^m(\B)\subseteq \mi(\Proj\A)^{{_1\bot_{\infty}}}$.

\item For every projective object $P$ in $\B$, we have
  $\pd_{\B}H(P)\leq m-1$ where the functor $H\colon \B \lxr \B$ is
  given in \eqref{endofunctorexactseq}.
\end{enumerate}
If $m=1$, then the above conditions are equivalent to the following one.
\begin{enumerate}[\rm(i)]
\item[\rm(iv)] The functor $\mi\colon \A\lxr \B$ is a homological
  embedding and the quotient functor $\me\colon \B\lxr \C$ preserves
  projective objects.
\end{enumerate}
\begin{proof}
  The equivalence of statements (i) and (ii) follows from the
  isomorphism
  $\Ext^n_{\B}(\mi(P),\Sigma^m(B))\simeq \Ext^{n+m}_{\B}(\mi(P),B)$ for
  all $n\geq 1$, $P$ in $\Proj\B$ and $B$ in $\B$.

  (ii)$\Longrightarrow$(iii) Let $B$ be an object in $\B$ and consider
  an injective coresolution in $\B\colon$
\[
 \xymatrix{
  0 \ar[r]^{ } & B \ar[r]^{ } & I^0 \ar[r]^{} &
  \cdots  \ar[r]^{} & I^{m-1} \ar[r] & \Sigma^m(B) \ar[r] & 0  }
\] 
Let $P$ be a projective object in $\B$. From our assumption and by
dimension shift we have
$\Ext^1_{\B}(\mi\mq(P),\Sigma^m(B))\simeq
\Ext^m_{\B}(\mi\mq(P),\Sigma(B))\simeq
\Ext^{m+1}_{\B}(\mi\mq(P),B)=0$. From $(\ref{endofunctorexactseq})$ there is an exact sequence 
\begin{equation}
\label{endofunctorexactseqprojective}
\xymatrix{
  0 \ar[r]^{ } & H(P) \ar[r]^{ } &
  P  \ar[r]^{ } & \mi\mq(P) \ar[r] & 0  }
\end{equation}
Applying the functor
$\Hom_{\B}(-,B)$ we get the long exact sequence:
\[
 \xymatrix@C=0.4cm{
  \cdots \ar[r]^{ } & \Ext^m_{\B}(P,B) \ar[r] & \Ext^m_{\B}(H(P),B) \ar[r]^{ } & \Ext^{m+1}_{\B}(\mi\mq(P),B) \ar[r]^{} & \Ext^{m+1}_{\B}(P,B)  \ar[r]^{} & \cdots   }
\] 
and therefore $\Ext^m_{\B}(H(P),B)=0$. We infer that $\pd_{\B}H
(P)\leq m-1$.

(iii)$\Longrightarrow$(ii) Let $B$ be an object in $\B$ and consider
an injective coresolution as above. From \cite[Remark 2.5]{Psaroud} we
have $\Proj\A=\add \mq(\Proj{\B})$, so it suffices to show that
$\Ext^n_{\B}(\mi\mq(P), \Sigma^m(B))=0$ for all $n\geq 1$ and any
object $P$ in $\Proj{\B}$. From the exact sequence $(\ref{endofunctorexactseqprojective})$ and our
assumption it follows that $\pd_{\B}\mi\mq(P)\leq m$. Then the result
follows easily as above by dimension shift.

Assume that $m=1$. We show (iii)$\Longrightarrow$(iv). Let $P$ be a
projective object in $\B$ and let $A$ be an object in $\A$. Consider
the exact sequence \eqref{endofunctorexactseqprojective} and apply the
functor $\Hom_{\B}(-,\mi(A))$.  Since $H(P)$ is projective, the long exact
sequence of homology yields that $\Ext^n_{\B}(\mi\mq(P),\mi(A))=0$ for
all $n\geq 2$. Furthermore, it shows that
$\Ext^1_{\B}(\mi\mq(P),\mi(A))$ is a quotient of
$\Hom_\B(H(P),\mi(A))$, which is isomorphic to
$\Hom_\A(\mq(H(P)),A)$ by the adjunction $(\mq,\mi)$.  Using that $H(P)$ is a quotient of
$\ml\me(P)$, the functor $\mq$ is right exact and $\mq\ml=0$, it
follows that $\mq(H(P)) = 0$.  Hence $\Ext^n_\B(\mi\mq(P),\mi(A)) =
0$ for all $n\geq 1$. Then \cite[Theorem 3.9]{Psaroud} implies that the functor $\mi$ is a homological embedding.

It remains to show that $\me(P)$ lies in $\Proj\C$. From
Proposition~\ref{properties} (iii) we have that the functor
$\me\colon \B\lxr \C$ is essentially surjective. Moreover, since the
composition $\me\mi=0$, see Proposition~\ref{properties} (ii),
if we apply the exact functor $\me$ to the sequence $(\ref{endofunctorexactseqprojective})$ we get that
$\me(H(P))\simeq \me(P)$. Thus, it suffices to show that
${\Ext}^{n}_{\C}(\me(H(P)),\me(Y))=0$ for any $n\geq 1$ and $Y$ in
$\B$. Since the functor $\mi$ is a homological embedding, it follows
from \cite[Corollary 3.11]{Psaroud} that there is an isomorphism
\begin{equation}
\label{extisomF(P)}
 \xymatrix{
  {\Ext}^n_{\B}(H(P),Y) \ar[r]^{\simeq \ \ \ \ } & {\Ext}^{n}_{\C}(\me(H(P)),\me(Y)) }
\end{equation}
for any $n\geq 1$. Since by (iii), the object $H(P)$ is projective it follows from the isomorphism $(\ref{extisomF(P)})$ that the object $\me(P)$ is projective.  

Finally, the implication (iv)$\Longrightarrow$(iii) follows immediately from the isomorphism $(\ref{extisomF(P)})$ using again that $\me(H(P))\simeq \me(P)$.
\end{proof}
\end{prop}

Let $\A$ be an abelian category with enough projective objects. Let
$A$ be an object in $\A$ and $\cdots \lxr P_1 \lxr P_0 \lxr A\lxr 0$ a
projective resolution of $A$ in $\A$. The kernel of the morphism
$P_{n-1}\lxr P_{n-2}$ is an $n$-th syzygy of $A$ and is denoted by
$\Omega^n(A)$.  Also, if $\X$ is a class of objects in $\A$, then we
denote by
${^{\bot}\X}=\{A\in \A \mid \Hom_{\A}(A,\X)=0\}$ the left
  orthogonal subcategory of $\X$.

We now recall some basic facts about projective covers in the setting of abelian categories. We refer to \cite{Krause:projectivecovers} for more details. Let $X$ be an object in an abelian category $\A$. The radical of $X$, denoted by $\mathsf{rad}X$, is the intersection of all its maximal subobjects. Clearly, if $S$ is a simple object then $\mathsf{rad}S=0$. Recall also that an epimorphism $f\colon P\lxr X$ is a projective cover, if $P$ is projective and the map $f$ is an essential epimorphism. In this case, the kernel $\Ker{f}$ is contained in the radical of $P$, i.e. $\Ker{f}\subseteq \mathsf{rad}P$. Moreover, given a morphism $g\colon X\lxr Y$ in $\A$ we always have $g(\mathsf{rad}X)\subseteq \mathsf{rad}Y$.

\begin{lem}
\label{lemisomorphismsimpleobject}
Let $\A$ be an abelian category with projective covers and let $S$
be a simple object. Then for every $n\geq 1$ there is an isomorphism
\[
{\Ext}^{n}_{\A}(X, S) \simeq \Hom_{\A}(\Omega^n(X), S)
\] 
\begin{proof}
Let $X$ be an object in $\A$. Consider a projective resolution of $X$ by taking projective covers
\begin{equation}
\label{projcoversresol}
 \xymatrix{
  \cdots \ar[r]^{ } & P_2 \ar[r]^{d_2} & P_1 \ar[r]^{d_1} &
  P_0  \ar[r]^{d_0} & X \ar[r] & 0  }
\end{equation}
This means that the map $d_0$ is a projective cover, the map
$P_1\lxr \Omega(X)$ is a projective cover and so on. Applying the
functor $\Hom_{\A}(-,S)$ we claim that the following complex
\begin{equation}
\label{homcomplexzerodifferential}
 \xymatrix{
  0 \ar[r]^{ } & \Hom_{\A}(P_0,S) \ar[r]^{(d_1)_*} & \Hom_{\A}(P_1,S) \ar[r]^{(d_2)_*} & \Hom_{\A}(P_2,S)  \ar[r]^{} & \cdots  }
\end{equation}
has zero differentials. By the construction of the resolution
$(\ref{projcoversresol})$ we have that
$\Image{d_{n}}=\Omega^n(X)\subseteq \mathsf{rad}P_{n-1}$ for all
$n\geq 1$. Let $f$ be a map in $\Hom_{\A}(P_n,S)$. Then
$(d_{n+1})_*(f)=f  d_{n+1}$ and we compute that
\[
\Image((d_{n+1})_*(f)) = f(\Image{d_{n+1}}) \subseteq
f(\mathsf{rad}P_n)\subseteq \mathsf{rad}f(P_n)\subseteq \mathsf{rad}S=0
\]
This implies that the complex $(\ref{homcomplexzerodifferential})$ has
zero differential and therefore we get that
\[
{\Ext}^{n}_{\A}(X, S) \simeq \Hom_{\A}(P_n, S)
\] 
Consider now the exact sequence $P_{n+1}\lxr P_n\lxr \Omega^n(X)\lxr 0$. Applying the functor $\Hom_{\A}(-,S)$ we obtain the exact sequence
\[
 \xymatrix{
  0 \ar[r]^{ } & \Hom_{\A}(\Omega^n(X),S) \ar[r]^{} & \Hom_{\A}(P_n,S) \ar[rr]^{0=(d_n)_*} && \Hom_{\A}(P_{n+1},S)   }
\]
and thus an isomorphism $\Hom_{\A}(\Omega^n(X),S)\simeq \Hom_{\A}(P_n,S)$.
This completes the proof of the desired isomorphism.
\end{proof}
\end{lem}

In the rest of this section we are interested in the eventual homological behaviour of the category $\B$ is a recollement $(\A,\B,\C)$. To this end the following result from \cite{PSS} is useful.

\begin{lem}\textnormal{(\!\!\cite[Theorem 3.4]{PSS})}
\label{lemmaThmPSS}  
Let $(\A, \B, \C)$ be a recollement of abelian categories with enough
projectives. The following statements are equivalent:
\begin{enumerate}[\rm(i)]
\item The object $B$ has a projective resolution of the form
\[
\xymatrix{
\cdots \ar[r] &
\ml(Q_1) \ar[r] &
\ml(Q_0) \ar[r] &
P_{n-1} \ar[r] &
\cdots \ar[r] &
P_0 \ar[r] &
B \ar[r] &
0
} 
\]
where each $Q_j$ is a projective object in $\C$.

\item $\Ext_{\B}^j(B,\mi(A))=0$ for every $A \in \A$ and $j > n$,
and there exists an $n$-th syzygy of $B$ lying in ${^{\bot}\mi(\A)}$.
\end{enumerate}
\end{lem}

Our aim is to provide a bound for the finitistic projective dimension
of an abelian category $\B$ in a recollement situation $(\A,\B,\C)$ when the $\A$-relative injective global
dimension ${\igld}_{\B}{\A}$ of $\B$ is finite.

Let $(\A, \B, \C)$ be a recollement of abelian categories. Let $B$ be
an object in $\B$ and $I$ an object in $\Inj\A$. Then by dimension
shift we have the isomorphism
$\Ext^n_{\B}(\Omega^m(B),\mi(I))\simeq \Ext_{\B}^{n+m}(B,\mi(I))$ for every
$n\geq 1$. This implies that ${\igld}_{\B}{\A}=t$ if and only if
$\Omega^t(\B)\subseteq {^{{_1\bot_\infty}}\mi(\Inj\A)}$. The latter is
equivalent to $\Omega^t(\B)\subseteq {^{{_1\bot_\infty}}\mi(\A)}$ by
\cite[Theorem 3.4]{PSS}. Under
the finiteness of the $\A$-relative injective global dimension of
$\B$, we show in the next result that each object admits a truncated
projective resolution after some step with projectives coming from the
quotient category $\C$ via the section functor $\ml$.

\begin{lem}
\label{lemtruncatedprojresol}
Let $(\A,\B,\C)$ be a recollement of abelian categories such that
$\A$ is a finite length category. Assume that ${\igld}_{\B}{\A}=t<\infty$
and that $\B$ has projective covers. Then any object of $\B$ has a
projective resolution as follows:
\[
\xymatrix{
\cdots \ar[r] &
\ml(Q_1) \ar[r] &
\ml(Q_0) \ar[r] &
P_{t} \ar[r] &
\cdots \ar[r] &
P_0 \ar[r] &
B \ar[r] &
0
} 
\]
where each $Q_j$ belongs to $\Proj\C$.
\begin{proof}
  From Lemma~\ref{lemmaThmPSS} it suffices to show that
  $\Ext^j_{\B}(B,\mi(A))=0$ for all $A$ in $\A$ and $j>t+1$, and there
  is an $(t+1)$-syzygy of $B$ lying in ${^{\bot}\mi(\A)}$. Since
  $\id_{\B}\mi(A)\leq t$, it follows that $\Ext^j_{\B}(B,\mi(A))=0$
  for all $j>t$. Since every object in $\A$ is filtered in finitely
  many steps by simple objects, it follows from
  Lemma~\ref{lemisomorphismsimpleobject} that
  $\Hom_{\B}(\Omega^{t+1}(B),\mi(A))$=0 for all $A$ in $\A$. We infer
  that $\Omega^{t+1}(B)$ lies in ${^{\bot}\mi(\A)}$ and this completes
  the proof.
\end{proof}
\end{lem}

We start by recalling from \cite{PSS} when a functor is an eventually
homological isomorphism. Note that we define the latter in the context
of abelian categories.

\begin{defn}
\label{defnevenhomiso}
Let $E\colon \B\lxr \C$ be a functor between abelian categories. The
functor $E$ is called an {\bf eventually homological isomorphism} if
there is a positive integer $t$ and a group isomorphism
\[
{\Ext}^n_{\B}(B,B')\simeq {\Ext}^n_{\C}(E(B),E(B'))
\]
for every $n>t$ and for all objects $B$, $B'$ in $\B$.  For the
minimal such $t$ it is called a $t$-eventually homological isomorphism.  
\end{defn}

Note that we do not require these isomorphisms to be induced by the
functor $E$, see Remark~\ref{remdefevenhomisom} below. We continue
with the next result from \cite{PSS} where a characterization is
provided for when the functor $\me$ in a recollement $(\A,\B,\C)$
induces isomorphisms of extension groups in almost all degrees.

\begin{thm}\textnormal{(\!\!\!\cite[Theorem 3.4]{PSS})}
\label{thm:ext-iso-proj}
Let $(\A,\B,\C)$ be a recollement of abelian categories and assume
that $\B$ and $\C$ have enough projective and injective objects.
Consider the following statements for an object $B$ of $\B$ and two
integers $n$ and $m$:
\begin{enumerate}[\rm(i)]
\item The map $\me_{B,B'}^{j} \colon \Ext_{\B}^{j}(B,B') \lxr
\Ext_{\C}^{j}(\me(B), \me(B'))$ is an isomorphism for every object $B'$ in
$\B$ and every integer $j > m + n$.

\item The object $B$ has a projective resolution of the form
\[
\ \ \ \ \  \xymatrix@C=0.7cm{
\cdots \ar[r] &
\ml(Q_1) \ar[r] &
\ml(Q_0) \ar[r] &
P_{n-1} \ar[r] &
\cdots \ar[r] &
P_0 \ar[r] &
B \ar[r] &
0
} 
\]
where each $Q_j$ lies in $\Proj\C$.

\item $\Ext_{\B}^j(B, \mi(A))=0$ for every $A$ in $\A$ and $j > n$,
  and there exists an $n$-th syzygy of $B$ lying in the left
  orthogonal subcategory ${^{\bot}\mi(\A)}$.

\item $\Ext_{\B}^j(B, \mi(I))=0$ for every $I$ in $\Inj{\A}$ and
  $j > n$, and and there exists an $n$-th syzygy of $B$ lying in the
  left orthogonal subcategory ${^{\bot}\mi(\Inj{\A})}$.
\end{enumerate}
Then the following statements are equivalent:$\textup{(ii)}
\Longleftrightarrow \textup{(iii)} \Longleftrightarrow
\textup{(iv)}$. If one of these holds and in addition we have
$\pd_{\C} \me(P) \le m$ for every projective object $P$ in
$\B$, then statement $\textup{(i)}$ holds.
\end{thm}

In \cite[Corollary 3.12]{PSS} the authors characterized when the
multiplication functor $e(-)\colon \smod\Lambda \lxr \smod{e\Lambda
  e}$ (see Example~\ref{examrecolmodcat}), where
$\Lambda$ is an Artin algebra, is an eventually homological
isomorphism. The following result generalizes \cite[Corollary
3.12]{PSS} from Artin algebras to abelian categories with certain
conditions.

\begin{prop}
\label{prop:algebra-ext-iso}
Let $(\A, \B, \C)$ be a recollement of abelian categories. Assume that
$\A$ is a finite length category and that $\B$ has projective
covers. The following statements are equivalent:
\begin{enumerate}[\rm(i)]
\item There is an integer $t$ such that for every pair of objects $X$
  and $Y$ in $\B$, and every $j > t$, the map
\[
\xymatrix{
\me^j_{X,Y} \colon {\Ext}_{\B}^j(X,Y) \ar[r]^{ \simeq \ } & \Ext^j_{\C}(\me(X),\me(Y)) }
\]
is an isomorphism.
\item The functor $\me\colon \B\lxr \C$ is an eventually homological isomorphism.
\item $(\alpha)$ ${\igld}_{\B}{\A}<\infty$
 and $(\beta)$ $\sup \{ \pd_\C \me(P) \mid P \in \Proj \B \} < \infty$.
\item $(\gamma)$ ${\pgld}_{\B}{\A}<\infty$ and $(\delta)$ $\sup \{ \id_\C \me(I) \mid I \in \Inj \B \} < \infty$.
\end{enumerate}
In particular, if the functor $e$ is a $s$-homological isomorphism,
then each of the dimensions in \textup{(iii)} and \textup{(iv)} are at
most $s$.  The bound $t$ in \textup{(i)} is bounded by the sum of
the dimensions occurring in \textup{(iii)}, and also bounded by the
sum of the dimensions occurring in \textup{(iv)}.
\begin{proof}
The implication (i)$\Longrightarrow$(ii) is clear.  

(ii)$\Longrightarrow$(iii): Let $A$ be an object of $\A$. Then, there
is a positive integer $t$ such that for every object $B$ in $\B$ and
for all $j>t$, we have the isomorphism
\[
{\Ext}_\B^j(B, \mi(A))\simeq {\Ext}_\C^j(\me(B), \me\mi(A))
\]
The latter extension group vanishes since $\me  \mi = 0$ by
Proposition~\ref{properties} (ii). This implies that
$\id_\B \mi(A) \le t$ and therefore $(\alpha)$ holds. For $(\beta)$,
let $P$ be a projective object of $\B$.  For every $C$ in $\C$ and
$j > t$, we have the following isomorphism
\[
{\Ext}_{\C}^j(\me(P), C)\simeq {\Ext}_{\C}^j(\me(P), \me\ml(C))
\simeq {\Ext}_\B^j(P, \ml(C))= 0
\]
since $\me \ml \simeq \iden_{\C}$ by Proposition~\ref{properties}
(iv). We infer that $\pd_\C \me(P) \le t$ and therefore statement
$(\beta)$ holds. Similarly we show that (ii)$\Longrightarrow$(iv).

(iii)$\Longrightarrow$(i): Assume that ${\igld}_{\B}{\A}=m-1<\infty$
and let $S$ be a simple object in $\A$. Since the object $\mi(S)$ is
simple in $\B$ by Lemma \ref{properties} (vi),
Lemma~\ref{lemisomorphismsimpleobject} provide us the 
isomorphism:
\[
{\Ext}_{\B}^j(B, \mi(S))\simeq \Hom_{\B}(\Omega^j(B), \mi(S))
\]
Since $\id_{\B}\mi(S)\leq m-1$, it follows that
$\Hom_{\B}(\Omega^j(B), \mi(S))=0$ for all $j>m-1$. Since every object
$A$ in $\A$ has a finite composition series, it follows that
$\Hom_{\B}(\Omega^j(B), \mi(A))=0$ for all $j>m-1$. So far we have
shown that $\Ext_{\B}^j(B, \mi(A))=0$ for every $A$ in $\A$ and
$j > m$, and there exists an $m$-th syzygy of $B$ lying in
${^{\bot}\mi(\A)}$. Using now assumption $(\beta)$,
Theorem~\ref{thm:ext-iso-proj} implies (i).

Similarly we show that (iv)$\Longrightarrow$(i) and then the four
statements are equivalent.
\end{proof}
\end{prop}

\begin{rem}\label{remdefevenhomisom}
\begin{enumerate}[\rm(i)]
\item The implication (ii)$\Longrightarrow$(i) in
  Proposition~\ref{prop:algebra-ext-iso} shows that whenever we know
  that the quotient functor $\me$ is an eventually homological
  isomorphism, we can obtain the desired isomorphisms from the functor
  $\me$. This explains why in Definition~\ref{defnevenhomiso} we don't
  require these isomorphisms to be induced by the involved functor.

\item It is natural to consider if any other pair of the four
  conditions $(\alpha)-(\delta)$ in
  Proposition~\ref{prop:algebra-ext-iso} implies that the functor
  $\me$ is an eventually homological isomorphism. We refer to
  \cite[Subsection 8.1]{PSS} for examples showing that this is not the
  case.
\end{enumerate}
\end{rem}

\section{Arrow removal and finitistic dimension}
\label{sec:cleft-findim}

In this and the following section the finitistic dimension of an
abelian category with enough projectives is discussed.  Recall that
given an abelian category $\A$ with enough projectives, the {\bf
  finitistic projective dimension} of $\A$ is defined by
\[
\Findim{\A}=\sup\{\pd{_{\A}A} \ | \ \pd{_{\A}A}<\infty\}.
\]
In the first subsection we begin by investigating cleft extensions of
abelian categories where the images of the endofunctors $F$ and $G$
have bounded projective dimension.  Such cleft extensions occur when
factoring out ideals generated by arrows not occurring in a set of
minimal relations (with some conditions) of finite dimensional
quotients of path algebras.  These special cleft extensions we study
and characterize in the second and the final subsection.

\subsection{Cleft extensions with special endofunctors $F$ and
    $G$}  
The exact sequences \eqref{splitsequencegivesF} and
\eqref{firstfundamentalsequence} give rise to the following exact
sequences of functors
\begin{equation}\label{firstexactsequence}
\xymatrix@C=0.5cm{
0 \ar[rr] && F \ar[rr]^{} && \me\ml \ar[rr]^{} && \iden_{\B}
\ar[rr] && 0  }
\end{equation}
\begin{equation}
\label{secondexactsequence}
\xymatrix@C=0.5cm{
0 \ar[rr] && G \ar[rr]^{} && \ml\me \ar[rr]^{} && \iden_{\A}
\ar[rr] && 0  }
\end{equation}
Consider the following  
\begin{equation}
\label{conditionsfortheImageFandG}
\Image{F}\subseteq \mathcal{P}^{n_{\B}}(\B) \ \ \ \ \text{and} \ \ \ \
\Image{G}\subseteq \mathcal{P}^{n_{\A}}(\A).
\end{equation}
This means that all objects in $\Image{F}$ and $\Image{G}$ have finite
projective dimension and there is a uniform bound for the length of
the shortest projective resolutions, which is $n_{\B}$ and $n_{\A}$, 
respectively. This is a general version of Theorem~A (ii) presented in the Introduction.

\begin{thm}\label{thm:cleftextfindim}
  Let $(\B,\A, \me, \ml, \mi)$ be a cleft extension $\A$ of $\B$,
  where $\A$ and $\B$ are abelian categories with enough projectives,
  such that condition $(\ref{assumptionscleftext})$ holds.  Assume in
  addition the conditions \eqref{conditionsfortheImageFandG}.
  Then:
\begin{enumerate}[\rm(i)]
\item $\Findim\A\leq \maxx\big\{\Findim\B, n_{\A}+1 \big\}$.
\item $\Findim\B\leq \maxx\big\{\Findim\A, n_{\B}+1 \big\}$.
\end{enumerate} 
In particular, 
\[
 \Findim\A < \infty \text{\ if and only if\ } \Findim\B < \infty.
\]
\end{thm}
\begin{proof}
  For any object $B$ in $\B$, using that the functors $\ml$ and $\me$
  are exact and preserve projective objects (by Lemma
  \ref{basicproperties} (iii) and (vii)), we have
\[
\pd_{\B}\me\ml(B) \leq \pd_{\A}\ml(B) \leq \pd_{\B}B.
\]
Similarly, for any object $A$ in $\A$ we have
\[
\pd_{\A}\ml\me(A) \leq \pd_{\B}\me(A) \leq \pd_{\A}A.
\]
Let $B$ be an object in $\B$ of finite projective dimension. Then we
obtain that $\pd_{\A}\ml(B)\leq\pd_\B B\leq \Findim\B$ since the
projective dimension of $B$ is finite. This implies that
$\pd_{\B}\me\ml(B)\leq \Findim\A$. Similarly, if $A$ is an object in
$\A$ of finite projective dimension, then we get that
$\pd_{\B}\me(A)\leq \Findim\A$ and therefore
$\pd_{\A}\ml\me(A)\leq \Findim\B$.  Then, if $\pd_{\B}B<\infty$, from
the exact sequence \eqref{firstexactsequence} it follows that
\[
\pd_{\B}B \leq \maxx \big\{\pd_{\B}\me\ml(B), \pd_{\B}F(B)+1
\big\} \leq \maxx \big\{ \Findim\A, n_{\B}+1 \big\}. 
\]
This implies that 
\begin{equation}
\label{firstinequality}
\Findim\B\leq \maxx\big\{\Findim\A, n_{\B}+1 \big\}. 
\end{equation}
Suppose now that $\pd_{\A}A<\infty$. Then from the exact sequence
\eqref{secondexactsequence} it follows that
\[
\pd_{\A}A \leq \maxx \big\{\pd_{\A}\ml\me(A),
\pd_{\A}G(A)+1 \big\} \leq \maxx \big\{ \Findim\B,
n_{\A}+1 \big\}, 
\]
and therefore we have 
\begin{equation}
\label{secondinequality}
\Findim\A\leq \maxx\big\{\Findim\B, n_{\A}+1 \big\}.
\end{equation}
From the relations \eqref{firstinequality} and
\eqref{secondinequality} it follows that $\Findim\A<\infty$ if and
only if $\Findim\B<\infty$.
\end{proof}

\subsection{Arrow removal}\label{sec:trivext} 
Now we apply the result from the previous subsection to quotients of
path algebras by admissible ideals, where we factor out an ideal
generated by arrows which do not occur in a minimal set of relations
with additional properties.

Let $\Lambda=kQ/I$ be an admissible quotient of a path algebra $kQ$.
If $a$ is an arrow in $Q$ which is not occurring in a minimal
generating set for $I$, then we refer to the quotient $\Gamma =
\Lambda/\langle a \rangle$ as \emph{arrow removal}.  In this subsection
we have two aims:
\begin{enumerate}[1)]
\item To show that arrow removal is a cleft extension
satisfying the conditions in Theorem \ref{thm:cleftextfindim}. 
\item Characterize arrow removal as trivial extensions with projective
  bimodules with special properties. 
\end{enumerate}

We start with the second goal. We remark that
Corollary~\ref{cor:trivialextwithcond} and
Proposition~\ref{prop:Lambda-a-Lambda} provide a general version of
Theorem~A (i), as presented in the Introduction, for removing multiple
arrows that do not occur in a minimal generating set of $I$.

Let $\Gamma = kQ^*/I^*$ be an admissible quotient of a path algebra
$kQ^*$.  First we analyse trivial extensions
$\Lambda = \Gamma\ltimes P$ with a projective
$\Gamma$-$\Gamma$-bimodule $P$ and show it is the same as adding one
arrow for each indecomposable direct summand of $P$ to $Q^*$ and
adding the relations $akQ^*b$ for all the different arrows $a$ and $b$
that we are adding.

\begin{prop}\label{prop:quiveroftrivext}
  Let $\Gamma = kQ^*/I^*$ for a field $k$, a finite quiver $Q^*$ and
  an admissible ideal $I^*$ in $kQ^*$.  Let $P$ be the projective
  $\Gamma^\evl$-module given by $P = \oplus_{i=1}^t\Gamma e_i\otimes_k
  f_i\Gamma$ for some trivial paths $e_i$ and $f_i$ in
  $Q^*$ for $i = 1,2,\ldots,t$, where $\Gamma^\evl =
  \Gamma^\op\otimes_k \Gamma$.  Then the trivial extension $\Lambda =
  \Gamma\ltimes P$ is isomorphic to a quotient $kQ/I$, where $Q$ has
  the same vertices and the same arrows as $Q^*$ with one arrow
  $a_i\colon v_{e_i}\lxr v_{f_i}$ added for each indecomposable summand
  $\Gamma e_i\otimes_k f_i\Gamma$ of $P$ and $I = \langle I^*,
  \sum_{i,j=1}^{t,t} a_ikQ^*a_j\rangle$ in $kQ$.  
\end{prop}
\begin{proof}
The radical $\rad(\Lambda)$ of $\Lambda$ is given by
$\rad(\Gamma)\ltimes P$, and therefore 
\[\rad(\Lambda)^2 =
\rad(\Gamma)^2\ltimes (\rad(\Gamma) P + P\rad(\Gamma)).\]
This gives that $\Lambda/\rad(\Lambda) \simeq kQ^*_0$ and
\begin{align}
\rad(\Lambda)/\rad(\Lambda)^2  & \simeq \rad(\Gamma)/\rad(\Gamma)^2 \ltimes
P/(\rad(\Gamma) P + P\rad(\Gamma))\notag\\
& \simeq kQ^*_1\ltimes \oplus_{i=1}^t (kQ^*_0 e_i\otimes_k
f_ikQ^*_0).\notag 
\end{align} 
It follows from this that the vertices in $Q$ are the same as the
vertices in $Q^*$ and that $Q$ has the same arrows as $Q^*$ and in
addition has one additional arrow $a_i$ for each indecomposable direct
summand $\Gamma e_i\otimes_k f_i\Gamma$ of $P$ corresponding to the
vectorspace $0\ltimes kQ^*_0e_i\otimes f_ikQ^*_0$.  The new arrow $a_i$
correspond to the element $x_i=(0, e_i\otimes f_i)$, which has the property
that
\[e_i x_i f_i = (0, e_i\cdot e_i \otimes f_i\cdot f_i) = (0, e_i\otimes f_i) = x_i.\] 
Hence $a_i\colon v_{e_i}\lxr v_{f_i}$. 

This shows that there is a surjective homomorphism of algebras
$\varphi\colon kQ\lxr \Lambda$ defined by
\[\varphi(x) = 
\begin{cases}
(x,0), \text{\ if $x\in Q_0^*\cup Q_1^*$},\\
(0, e_i\otimes f_i), \text{\ if $x = a_i$ for $i=1,2,\ldots, t$}.
\end{cases}\]
It is clear that $\langle I^*\rangle$ and
$\sum_{i,j=1}^{t,t}a_ikQ^*a_j$ are in the kernel of $\varphi$, since
$(0\ltimes P)^2 = 0$.  We want to show that 
\[\Ker\varphi = \langle I^*, \sum_{i,j=1}^{t,t}a_ikQ^*a_j\rangle.\]

Let $z$ be an arbitrary element in the ideal generated by the arrows
in $kQ$ and in the kernel of $\varphi$.  Given that we know
$\langle I^*, \sum_{i,j=1}^{t,t}a_ikQ^*a_j\rangle$ is in the kernel of
$\varphi$, we only need to consider elements $z$ of the form
\[z = r + \sum_{i,j=1}^{t,s_i} r_{1ij} a_i r_{2ij}.\] 
for $r$, $r_{1ij}$ and $r_{2ij}$ in $kQ^*$, where $r_{1ij} = r_{1ij}e_i$ and
$r_{2ij} = f_ir_{2ij}$ for $1\leqslant j \leqslant s_i$.  By the
algebra homomorphism $\varphi$ this element is mapped to
\[0 = (r, \sum_{i,j=1}^{t,s_i} r_{1ij}x_i r_{2ij}).\] Then we must
have that $r = 0$ in $\Gamma$ and $\sum_{i,j=1}^{t,s_i} r_{1ij}x_i
r_{2ij} = 0$ in $P$.  We have that
\[r_{1ij}x_ir_{2ij} = r_{1ij}(0, e_i\otimes f_i)r_{2ij} = (0, r_{1ij}e_i\otimes
f_ir_{2ij}) = (0, r_{1ij}\otimes r_{2ij}),\]
and this gives that 
\[
\sum_{i,j=1}^{t,s_i} r_{1ij} x_ir_{2ij} = \sum_{i,j=1}^{t,s_i} (0,r_{1ij}\otimes r_{2ij})
 = (0,\sum_{i=1}^t  \sum_{j=1}^{s_i}r_{1ij}\otimes r_{2ij})
\]
and in particular that $\sum_{j=1}^{s_i} r_{1ij}\otimes
r_{2ij} = 0$ in $\Gamma e_i\otimes_k f_i\Gamma$ for each $i$.  Since 
\[\Gamma e_i\otimes_k f_i\Gamma \simeq kQ^*e_i\otimes_k f_ikQ^*/(I^*e_i\otimes_k
  f_ikQ^* + kQ^*e_i\otimes_k f_iI^*),\]
it follows from the above that kernel of $\varphi$ is $\langle I^*,
\sum_{i,j=1}^{t,t} a_ikQ^*a_j\rangle$ in $kQ$.  This completes the proof of the
proposition.
\end{proof}

Now we prove that a trivial extension of an algebra $\Gamma$ with a
finitely generated projective bimodule $P$ with certain properties
gives rise to an algebra where arrows can be removed.  This is the
first part of the characterization of arrow removal. 

\begin{cor}\label{cor:trivialextwithcond}
  Let $\Gamma = kQ^*/I^*$ for a field $k$, a finite quiver $Q^*$ and
  an admissible ideal $I^*$ in $kQ^*$.  Let $P$ be the projective
  $\Gamma^\evl$-module given by
  $P = \oplus_{i=1}^t\Gamma e_i\otimes_k f_i\Gamma$ for some trivial
  paths $e_i$ and $f_i$ in $Q^*$ for $i = 1,2,\ldots,t$, where
  $\Gamma^\evl = \Gamma^\op\otimes_k \Gamma$.  Suppose that
  $\Hom_\Gamma( e_i\Gamma, f_j\Gamma) = 0$ for all $i$ and $j$ in
  $\{1,2,\ldots,t\}$.
\begin{enumerate}[\rm(i)]
\item The trivial extension $\Lambda = \Gamma\ltimes P$ is isomorphic
  to a quotient $kQ/I$, where $Q$ has the same vertices and the same
  arrows as $Q^*$ with one arrow $a_i\colon v_{e_i}\lxr v_{f_i}$ added
  for each indecomposable summand $\Gamma e_i\otimes_k f_i\Gamma$ of
  $P$ and $I = \langle I^*\rangle$ in $kQ$.  Here $v_{e_i}$ and
  $v_{f_i}$ are the vertices corresponding to the primitive
  idempotents $e_i$ and $f_i$, respectively.  
\item The arrows $a_i$ do not occur in a minimal set of generators for
  the relations $I$, and $\Hom_\Lambda( e_i\Lambda, f_j\Lambda) = 0$
  for all $i$ and $j$ in $\{1,2,\ldots,t\}$.
\end{enumerate}
\end{cor}
\begin{proof}
  (i) We have that
  $0 = \Hom_\Gamma( e_i\Gamma, f_j\Gamma) \simeq f_j\Gamma e_i$ for
  all $i$ and $j$, which is equivalent to saying that $f_jkQ^*e_i$ is
  contained in $I^*$ for all $i$ and $j$.  We infer from this that
  $ a_jkQ^*a_i = a_j(f_jkQ^*e_i)a_i$ is in the ideal
  $\langle I^*\rangle$ in $kQ$.  The claim then follows immediately
  from Proposition \ref{prop:quiveroftrivext}.

  (ii) The first claim follows directly from (i).  We have that
\begin{align}
\Hom_\Lambda(e_i\Lambda, f_j\Lambda) & \simeq f_j\Lambda e_i\notag\\
& = f_j(\Gamma\ltimes P)e_i\notag\\
& = (f_j\Gamma e_i)\ltimes (f_jPe_i)\notag\\
& = (f_j\Gamma e_i)\ltimes \oplus_{l=1}^t (f_j\Gamma e_l\otimes_k
f_l\Gamma e_i)\notag
\end{align}
By assumption $f_r\Gamma e_s = 0$ for all $r$ and $s$ in
$\{1,2,\ldots,t\}$, so we obtain that 
\[\Hom_\Lambda( e_i\Lambda, f_j\Lambda ) = 0\]
for all $i$ and $j$ in $\{1,2,\ldots,t\}$.  This completes the proof
of the proposition.
\end{proof}

Next we prove the converse of the above result.  This needs some
preparation.  Let $\Lambda = kQ/I$ be an admissible quotient of the
path algebra $kQ$ over a field $k$.  Suppose that there are arrows
$a_i\colon v_{e_i}\lxr v_{f_i}$ in $Q$ for $i=1,2,\ldots,t$ which do
not occur in a set of minimal generators of $I$ in $kQ$ and
$\Hom_\Lambda( e_i\Lambda, f_j\Lambda ) = 0$ for all $i$ and $j$ in
$\{1,2,\ldots,t\}$.  Let $\Gamma = \Lambda/\Lambda
\{\overline{a}_i\}_{i=1}^t \Lambda$ for $\overline{a}_i = a_i + I$ in
$\Lambda$.  We have the natural surjective algebra homomorphism
$\pi\colon \Lambda \lxr \Gamma$, and we claim that there is a natural
algebra inclusion $\nu\colon \Gamma\hookrightarrow \Lambda$ such that $\pi \nu =
\id_\Gamma$.  Let $Q^*$ be the subquiver of $Q$, where the arrows
$\{a_i\}_{i=1}^t$ have been removed.  The quiver inclusion morphism
$Q^*\lxr Q$ induces an inclusion $kQ^*\to kQ$ of path algebras. This
further induces an inclusion
\[\nu' \colon kQ^*/(kQ^* \cap I) \to kQ/I = \Lambda.\]
We want to show that $\Gamma \simeq kQ^*/(kQ^* \cap I)$.  We have that
\begin{align}
\Gamma = \Lambda/\Lambda\{\overline{a}_i\}_{i=1}^t\Lambda & =
\left(kQ/I\right)/\left((kQ/I)\{\overline{a}_i\}_{i=1}^t(kQ/I)\right)\notag\\
& = \left(kQ/I\right)/\left(kQ\{a_i\}_{i=1}^tkQ + I\right)/I\notag\\
& \simeq kQ/(kQ\{a_i\}_{i=1}^tkQ + I)\notag
\end{align}
Furthermore, 
\[kQ = kQ\{a_i\}_{i=1}^tkQ + kQ^*,\] where the sum is direct as
vectorspaces, hence $kQ = kQ\{a_i\}_{i=1}^tkQ \oplus kQ^*$.  In
addition, since $I$ is generated by $I\cap kQ^*$, it follows that
\[kQ\{a_i\}_{i=1}^tkQ + I = kQ\{a_i\}_{i=1}^tkQ + (I\cap kQ^*)\]
in $kQ$. As above,
\[kQ\{a_i\}_{i=1}^tkQ + I = kQ\{a_i\}_{i=1}^tkQ \oplus (I\cap kQ^*),\] 
and this implies that 
\begin{align}
\Gamma \simeq kQ/(kQ\{a_i\}_{i=1}^tkQ + I) & = (kQ\{a_i\}_{i=1}^tkQ
\oplus kQ^*)/(kQ\{a_i\}_{i=1}^tkQ \oplus (I\cap kQ^*))\notag\\
& \simeq kQ^*/(I\cap kQ^*).\notag
\end{align}
We infer from this that the epimorphism $\varphi\colon kQ^*/(I\cap
kQ^*) \to \Lambda/\Lambda\{a_i\}_{i=1}^t\Lambda = \Gamma$ given by
$\varphi( \overline{p} ) = \overline{ p + I}$ for $p$ in $kQ^*$ is an
isomorhism.  The inclusion $kQ^*\to kQ$ induces an inclusion
$\nu\colon kQ^*/(I\cap kQ^*) \to \Lambda$ in such a way that the
composition $\varphi^{-1}\pi\nu = \id$.  If we now identify $\Gamma$
with $kQ^*/(I\cap kQ^*)$, we have our desired result.

The exact sequence
\begin{equation}\label{eq:exact-eta}
\eta\colon 0\to \Lambda \{\overline{a}\}_{i=1}^t\Lambda \to \Lambda
\to \Gamma \to 0 
\end{equation}
can be considered as a sequence of $\Gamma$-$\Gamma$-bimodules. This
sequence splits as an exact sequence of one-sided $\Gamma$-modules, so
that $\Lambda \simeq \Lambda \{\overline{a}\}_{i=1}^t\Lambda \oplus
\Gamma$ as a left and as a right $\Gamma$-module.  Next we prove that
$\Lambda \{\overline{a}\}_{i=1}^t\Lambda$ is a projective
$\Gamma$-$\Gamma$-bimodule and that we have the converse of Corollary
\ref{cor:trivialextwithcond}.

\begin{prop}\label{prop:Lambda-a-Lambda}
  Let $\Lambda = kQ/I$ be an admissible quotient of the path algebra
  $kQ$ over a field $k$.  Suppose that there are arrows $a_i\colon
  v_{e_i}\to v_{f_i}$ in $Q$ for $i=1,2,\ldots,t$ which do not occur
  in a set of minimal generators of $I$ in $kQ$ and $\Hom_\Lambda(
  e_i\Lambda, f_j\Lambda ) = 0$ for all $i$ and $j$ in
  $\{1,2,\ldots,t\}$.  Let $\Gamma = \Lambda/\Lambda
  \{\overline{a}_i\}_{i=1}^t \Lambda$.  Then the following assertions
  are true.
\begin{enumerate}[\rm(i)]
\item $f_j\Gamma e_i \simeq \Hom_\Gamma(e_i\Gamma, f_j\Gamma) = 0$ for
  all $i$ and $j$ in $\{1,2,\ldots,t\}$.  
\item \[\sum_{i=1}^t\Lambda \overline{a}_i \Lambda \simeq
  \oplus_{i=1}^t\Lambda\overline{a}_i\Lambda\]
and
\[\Lambda\overline{a}_i\Lambda \simeq \Gamma e_i\otimes_k f_i\Gamma\] 
as $\Gamma$-$\Gamma$-bimodules.  In particular,
$\sum_{i=1}^t\Lambda\overline{a}_i\Lambda$ is a projective
$\Gamma$-$\Gamma$-bimodule.
\item $\Lambda$ is isomorphic to the trivial extension $\Gamma\ltimes
  P$, where $P = \oplus_{i=1}^t \Gamma e_i\otimes_k f_i\Gamma$ with
  $\Hom_\Gamma(e_i\Gamma, f_j\Gamma) = 0$ for all $i$ and $j$ in
  $\{1,2,\ldots,t\}$. 
\end{enumerate}
\end{prop}
\begin{proof}
(i) Since $f_j\Lambda e_i = 0$ for all $i$ and $j$ in
$\{1,2,\ldots,t\}$ and $\nu\colon \Gamma \to \Lambda$ is an inclusion,
it follows that $0 = f_j\Gamma e_i \simeq \Hom_\Gamma(e_i\Gamma,
f_j\Gamma)$ for all $i$ and $j$ in $\{1,2,\ldots,t\}$. 

(ii) First we argue that $\sum_{i=1}^t\Lambda\overline{a}_i\Lambda$ is
a direct sum.  Since $f_j\Lambda e_i = 0$ for all $i$ and $j$ in
$\{1,2,\ldots,t\}$, it follows that $a_j\Lambda a_i = 0$ for all $i$
and $j$ in $\{1,2,\ldots,t\}$.  This implies that
\[\Lambda\overline{a}_i\Lambda = \Gamma\overline{a}_i\Gamma\]
as a $\Gamma$-$\Gamma$-bimodule.  Let 
\[x =
\sum_{i=1}^t\sum_{j=1}^{t_i}c_{ji}\lambda_{ji}\overline{a}_i\lambda'_{ji}\]
be in $\sum_{i=1}^t\Lambda\overline{a}_i\Lambda$, where $c_{ji}$ is in
$k \setminus \{0\}$, and the elements $\lambda_{ji}$ and $\lambda_{ji}'$
are in $\Nontip(I)$ by Lemma \ref{lem:normalform} (iii).  Assume that
$x = 0$, or equivalently, when $x$
is viewed as an element in $kQ$, then $x$ is in $I$.  Then the tip of
$x$ is divisible by a tip of an element of a Gr\"obner basis $\G$ for
$I$ in $kQ$ by the definition of a Gr\"obner basis (see Definition
\ref{defn:Groebner}). The tip of $x$ is of the form $\lambda_{j_0i_0} 
a_i\lambda'_{j_0i_0} = p\Tip(g)p'$ for some integers $i_0$ and $j_0$,
some paths $p$ and $p'$ and some $g$ in $\G$ by Lemma
\ref{lem:existencegroebner} (iii).  Since the 
arrows $a_i$ do not occur in a set of minimal generators for $I$, the
element $\Tip(g)$ must divide $\lambda_{j_0i_0}$ or
$\lambda'_{j_0i_0}$ by Lemma \ref{lem:groebneravoidingarrow}.  Since
$\lambda_{j_0i_0}$ and $\lambda'_{j_0i_0}$
are in $\Nontip(I)$, this is a contradiction.  Hence, we must have
$c_{j_0i_0} = 0$, which is another contradiction.  It follows that the
sum $\sum_{i=1}^t\Lambda\overline{a}_i\Lambda$ is direct. 

Now we show that $\Lambda\overline{a}_i\Lambda \simeq \Gamma
e_i\otimes_k f_i\Gamma$ for all $i$ in $\{1,2,\ldots,t\}$.  This shows
that $\sum_{i=1}^t\Lambda\overline{a}_i\Lambda$ is a projective
$\Gamma$-$\Gamma$-bimodule.

Consider the map 
\[\psi\colon \Gamma e_i\otimes_k f_i\Gamma \to
\Lambda\overline{a}_i\Lambda\]
given by $\psi(\gamma e_i\otimes f_i\gamma') = \gamma
e_i\overline{a}_if_i\gamma'$ for $\gamma e_i\otimes f_i\gamma'$ in
$\Gamma e_i\otimes_k f_i\Gamma$.  Any element $x$ in $\Gamma
e_i\otimes_k f_i\Gamma$ can be written as
\[x = \sum_{r=1}^n \alpha_r(\gamma_re_i\otimes f_i\gamma_r')\]
with $\alpha_r$ in $k\setminus \{0\}$ and $\gamma_r$ and $\gamma_r'$
in $\Nontip(I)\cap kQ^*$ by Lemma \ref{lem:nontipsarrowremoved}.  If
$x$ is in $\Ker\psi$, then 
\[\psi(x) = \sum_{r=1}^n \alpha_r
\gamma_re_i\overline{a}_if_i\gamma_r' = 0\]
in $\Lambda\overline{a}_i\Lambda$, or equivalently that $\psi(x)$ is
in $I$.  The tip of $\psi(x)$ is $\Tip(\psi(x)) =
\gamma_{r_0}a_i\gamma_{r_0}'$ for some $r_0$.  As $\psi(x)$ is in $I$,
this tip must be divisible by some tip of an elemenet $g$ in $\G$,
a Gr\"obner basis for $I$.  The arrow $a_i$ does not occur in any
element in $\G$ by Lemma \ref{lem:groebneravoidingarrow}, so we infer
that $\Tip(g)$ divides $\gamma_{r_0}$ 
or $\gamma_{r_0}'$.  But this is impossible, since $\gamma_{r_0}$ and
$\gamma_{r_0}'$ are elements in $\Nontip(I)$.  It follows that
$\alpha_{r_0}= 0$, which is a contradiction.  Hence we have that
$\Ker\psi = 0$ and $\Gamma e_i\otimes_k f_i\Gamma \simeq
\Lambda\overline{a}_i\Lambda$, which is a projective
$\Gamma$-$\Gamma$-bimodule.  This completes the proof of (ii).  

(iii) This follows from the comments before this proposition and parts
(i) and (ii). 
\end{proof}

Now we look at the special case removing only one arrow
$a\colon v_e\lxr v_f$ not occurring in a set of minimal generators of
$I$ and $\Hom_\Lambda(e\Lambda, f\Lambda) = 0$.  We show that the
second condition is superfluous. This is Theorem~A (i) as stated in
the Introduction. 

\begin{prop}\label{prop:onearrowchar}
  Let $\Lambda = kQ/I$ be an admissible quotient of a path algebra
  $kQ$ over a field $k$.  Then an arrow $a\colon v_e\lxr v_f$ in $Q$
  does not occur in a set of minimal generators of $I$ in $kQ$ if and
  only if $\Lambda$ is isomorphic to the trivial extension $\Gamma\ltimes
  P$, where $\Gamma \simeq \Lambda/\Lambda\overline{a}\Lambda$ and $P
  = \Gamma e\otimes_kf\Gamma$ with $\Hom_\Gamma(e\Gamma,
  f\Gamma)=(0)$.  
\end{prop}
\begin{proof}
  Assume that the arrow $a\colon v_e\lxr v_f$ in $Q$ does not occur in
  a set of minimal generators of $I$ in $kQ$.  Given the assumption on
  $a$, a minimal set of generators for a Gr\"obner
  basis $\G$ for the ideal $I$ (using length left-lexicographic
  ordering on the paths in $Q$) does not contain any elements in which
  the arrow $a$ appears by Lemma~\ref{lem:groebneravoidingarrow}.

  First we show that $\Hom_\Lambda(e\Lambda, f\Lambda) = (0)$.  We
  have that $\Hom_\Lambda(e\Lambda, f\Lambda) \simeq f\Lambda e$.
  Since $a$ is not occuring in set of minimal generators of $I$ in
  $kQ$, the multiplication map $f\Lambda e\lxr \overline{a}\Lambda e$
  given by left multiplication by $a$, is an isomorphism.  Hence,
  $f\Lambda e = 0$ if and only if $\overline{a}\Lambda e = 0$.

  Assume that $\overline{a}\Lambda e \neq 0$, that is, there is some
  path $p$ in $Q$ such that $\overline{ap}\neq 0$ and $p$ ends in
  $v_e$.  Since $I$ is an admissible ideal, we have that
  $(\overline{ap})^t = 0$ in $\Lambda$ or equivalently $(ap)^t$ is in
  $I$ for some $t\geqslant 1$.  In particular, reducing the element
  $(ap)^t$ modulo a minimal set of generators for a Gr\"obner basis
  for $I$ gives zero.  Choose $p$ minimal with the property that
  $ap\not\in I$.  Reducing $(ap)^t$ modulo $I$ means to substract
  elements of the form $p'gp''$ for some paths $p'$ and $p''$ in $Q$
  and $g$ in $\G$, where
\[\Tip((ap)^t) = (ap)^t = \Tip(p'gp'') = p'\Tip(g)p''.\]
Since $\Tip(g)$ does not contain $a$ by Lemma
\ref{lem:groebneravoidingarrow}, we must have that $\Tip(g)\mid
p$, where this means $\Tip(g)$ divides $p$, that is, $q\Tip(g)q' = p$
for some path $q$ and $q'$ in $Q$.  Hence we can write
\begin{align}
p & = q\Tip(g)q'\notag\\
p' & = (ap)^raq\notag\\
p'' & = q'(ap)^{t-r-1}.\notag
\end{align}
Then 
\begin{align}
(ap)^t - p'gp'' & = (ap)^ra(p -
qgq')(ap)^{t-r-1}\notag\\
& = (ap)^ra(p-\Tip(qgq') - \{\sum\textrm{smaller paths than
  $p$}\}(ap)^{t-r-1}\notag\\
& = - (ap)^ra\{\sum\textrm{smaller paths than
  $p$}\}(ap)^{t-r-1}.\notag
\end{align}
Since for all the paths $s$ that are smaller than $p$ the elements
$as$ are in $I$, it follows from the above that $ap$ is in $I$.  This
is a contradiction, so $\overline{a}\Lambda e = 0$.  Hence
$f\Lambda e = 0$ and $\Hom_\Lambda(e\Lambda, f\Lambda) = 0$. 

From the above and Proposition \ref{prop:Lambda-a-Lambda}
it follows that $\Lambda$ is isomorphic to the trivial extension
$\Gamma\ltimes P$, where
$\Gamma \simeq \Lambda/\Lambda\overline{a}\Lambda$ and
$P = \Gamma e\otimes_kf\Gamma$ with
$\Hom_\Gamma(e\Gamma, f\Gamma) = (0)$.

Conversely, assume that $\Lambda$ is isomorphic to the trivial
extension $\Gamma\ltimes P$, where
$\Gamma \simeq \Lambda/\Lambda\overline{a}\Lambda$ and
$P = \Gamma e\otimes_kf\Gamma$ with
$\Hom_\Gamma(e\Gamma, f\Gamma) = (0)$.  Using Corollary
\ref{cor:trivialextwithcond} the claim follows. 
\end{proof}

Let the setting be as above, $\Lambda = kQ/I$ be an admissible
quotient of the path algebra $kQ$ over a field $k$.  Suppose that
there are arrows $a_i\colon v_{e_i}\lxr v_{f_i}$ in $Q$ for
$i=1,2,\ldots,t$ which do not occur in a set of minimal generators of
$I$ in $kQ$ and $\Hom_\Lambda( e_i\Lambda, f_j\Lambda ) = 0$ for all
$i$ and $j$ in $\{1,2,\ldots,t\}$.  Let $\Gamma = \Lambda/\Lambda
\{\overline{a}_i\}_{i=1}^t \Lambda$.  Now we want to address the first
aim of this section, namely that arrow removal $\Lambda\lxr \Gamma$ is
a cleft extension satisfying the conditions in 
Theorem~\ref{thm:cleftextfindim}.  

In the above situation we have the functors
\begin{equation}
\label{digramwithendofunctors}
\xymatrix@C=0.5cm{
\smod\Gamma  \ar@(ul,dl)_{F} \ar[rrr]^{\mi = \Hom_\Gamma({_\Lambda\Gamma}_\Gamma,-)} &&& \smod\Lambda
\ar@(ul,ur)^{G} \ar[rrr]^{\me = \Hom_\Lambda({_\Gamma\Lambda}_\Lambda,-)} 
\ar @/_1.5pc/[lll]_{\mq = {-\otimes_\Lambda  {_\Lambda\Gamma}_\Gamma}}  \ar
 @/^1.5pc/[lll]^{\map = \Hom_\Lambda({_\Gamma\Gamma}_\Lambda,-)} &&& \smod\Gamma \ar@(ur,dr)^{F}
\ar @/_1.5pc/[lll]_{\ml = {-\otimes_\Gamma {_\Gamma\Lambda}_\Lambda}} \ar
 @/^1.5pc/[lll]^{\mr = \Hom_\Gamma({_\Lambda\Lambda}_\Gamma,-)} }
\end{equation}
where $F$ is given by the exact sequence
\[
\xymatrix{ 0 \ar[r] & F \ar[r] & \me\ml \ar[r] & \iden_{\smod\Gamma}}.
\]

\begin{prop} 
\label{propremoveiscleft}
Let $\Lambda = kQ/I$ be an admissible quotient of the path algebra
$kQ$ over a field $k$.  Suppose that there are arrows $a_i\colon v_{e_i}\lxr v_{f_i}$ in $Q$ for $i=1,2,\ldots,t$ which do not occur in a set of minimal generators of $I$ in $kQ$ and $\Hom_\Lambda(e_i\Lambda, f_j\Lambda ) = 0$ for all $i$ and $j$ in $\{1,2,\ldots,t\}$.  Let $\Gamma = \Lambda/\Lambda
\{\overline{a}_i\}_{i=1}^t \Lambda$.  Then the following assertions hold. 
\begin{enumerate}[\rm(i)]
\item $\me$ is faithful exact,
\item $(\ml,\me)$ is an adjoint pair of functors,
\item $\me\mi \simeq 1_{\smod\Gamma}$,
\item $\ml$ and $\mr$ are exact functors,
\item $\me$ preserves projectives, 
\item $\Image F \subseteq \proj(\Gamma)$. 
\item $F^2 = 0$. 
\end{enumerate}
In particular, 
\[\C=(\smod\Gamma, \smod\Lambda, \me = \Hom_\Lambda(
  {_\Gamma\Lambda_\Lambda}, -), \ml = -\otimes_\Gamma
  {_\Gamma\Lambda_\Lambda}, \mi =
  \Hom_\Gamma({_\Lambda\Gamma_\Gamma},-))\] is a cleft extension
$\smod\Lambda$ of $\smod\Gamma$ satisfying \eqref{assumptionscleftext} and 
\eqref{conditionsfortheImageFandG} with both bounds being zero, that is, $\Image F \subseteq \Proj(\Gamma)$ and $\Image G\subseteq \Proj(\Lambda)$.
\end{prop}
\begin{proof}
(i) The functor $\me\colon \smod\Lambda\lxr \smod\Gamma$ is faithful
exact, since it is given by the restriction along the algebra
inclusion $\Gamma\lxr \Lambda$.

(ii) This is immediate from the definitions of the functors $\ml$ and
$\me$. 

(iii) Since the composition of the algebra homomorphisms $\nu\colon
\Gamma \lxr \Lambda$ and $\pi\colon \Lambda\lxr \Gamma$ is
the identity on $\Gamma$, it follows that $\me \mi \simeq
\iden_{\smod\Gamma}$. 

(iv) By the split exact sequence \eqref{eq:exact-eta} as left
$\Gamma$-modules, we have the isomorphism $_\Gamma\Lambda\simeq {_\Gamma\Gamma}\oplus {_\Gamma\Lambda\{\overline{a}_i\}_{i=1}^t\Lambda}$.  By Proposition~\ref{prop:Lambda-a-Lambda} (ii) we have that 
\[{_\Gamma\Lambda\{\overline{a}_i\}_{i=1}^t\Lambda} \simeq
\oplus_{i=1}^t {_\Gamma\Gamma e_i}\otimes_k f_i\Gamma \simeq
\oplus_{i=1}^t \Gamma e_i^{\dim_kf_i\Gamma},
\]
so that $_\Gamma\Lambda$ is a projective left $\Gamma$-module.  Since
\[\ml= - \otimes_\Gamma {_\Gamma\Lambda}_\Lambda\colon \smod\Gamma\lxr
\smod\Lambda,\]
the functor $\ml$ is exact.  Using similar arguments as above we
  show that $\Lambda_\Gamma$ is a projective $\Gamma$-module, hence
  $\mr = \Hom_\Gamma( {_\Lambda\Lambda_\Gamma},-)$ is an exact functor.

(v) Since $(\me,\mr)$ is an adjoint pair and $\mr$ is exact by (iv), it follows that $\me$ preserves projectives.

(vi) Since $\me$ and $\ml$ commutes with finite direct sums and they
are exact, we infer that $F$ also commutes with finite direct sums and
is exact.  By Watt's theorem
$F \simeq -\otimes_\Gamma F(\Gamma)\colon \smod\Gamma\lxr \smod\Gamma$,
where $F(\Gamma) \simeq \Lambda \{\overline{a}_i\}_{i=1}^t \Lambda$ as
$\Gamma$-$\Gamma$-bimodules.   We have that 
\[F(B) = B\otimes_\Gamma {_\Gamma\Lambda}
  \{\overline{a}_i\}_{i=1}^t\Lambda_\Gamma\simeq
  \oplus_{i=1}^tB\otimes_\gamma {_\Gamma\Gamma}e_i \otimes_k
  f_i\Gamma_\Gamma\simeq \oplus_{i=1}^tBe_i\otimes_kf_i\Gamma\] by
Proposition \ref{prop:Lambda-a-Lambda} (ii).  The claim follows from
this.

(vii) We have that 
\[F^2(B) = B\otimes_\Gamma \Lambda\{\overline{a}_i\}_{i=1}^t\Lambda\otimes_\Gamma \Lambda
\{\overline{a}_i\}_{i=1}^t\Lambda.\] 
Since $\overline{a}_j\Lambda = \overline{a}_j\Gamma =
\overline{a}_jf_j\Gamma$ and 
$\Lambda \overline{a}_i = \Gamma\overline{a}_i = \Gamma
e_i\overline{a}_i$, we have that 
\[B\otimes_\Gamma \Lambda\overline{a}_j\Lambda\otimes_\Gamma
  \Lambda\overline{a}_i\Lambda = 
B\otimes_\Gamma \Lambda\overline{a}_jf_j\Gamma\otimes_\Gamma
  \Gamma e_i\overline{a}_i\Lambda =
B\otimes_\Gamma \Lambda\overline{a}_j\underbrace{f_j\Gamma e_i}_{=0}\otimes_\Gamma
  \overline{a}_i\Lambda.\]

From this we infer that $F^2(B) = 0$ as $f_j\Gamma e_i = 0$ by
  Proposition \ref{prop:Lambda-a-Lambda} (i) for all $i$ and $j$ in $\{1,2,\ldots, t\}$.

For the final claim, (i)--(iii) show that $\C$ is a cleft extension. Conditions (iv) and (v) show that \eqref{assumptionscleftext} holds. 

By (vi) we have that $\Image F \subseteq \Proj(\Gamma)$.  Therefore
applying (vii) and Lemma~\ref{lem:FGequivalence} we have
$\Image G \subseteq \Proj(\Lambda)$. This shows that \eqref{conditionsfortheImageFandG} holds with 
$n_{\smod\Gamma} = n_{\smod\Lambda} = 0$. 
\end{proof}

We end this section with the following consequence of the previous
result and Theorem \ref{thm:cleftextfindim}. This is Theorem~A (ii) presented in the Introduction.

\begin{thm}\label{thm:summarycleftext}
Let $\Lambda = kQ/I$ be an admissible quotient of the path algebra
  $kQ$ over a field $k$.  Suppose that there are arrows $a_i\colon
  v_{e_i}\lxr v_{f_i}$ in $Q$ for $i=1,2,\ldots,t$ which do not occur
  in a set of minimal generators of $I$ in $kQ$ and $\Hom_\Lambda(
  e_i\Lambda, f_j\Lambda ) = 0$ for all $i$ and $j$ in
  $\{1,2,\ldots,t\}$.  Let $\Gamma = \Lambda/\Lambda
  \{\overline{a}_i\}_{i=1}^t \Lambda$.
\begin{enumerate}[\rm(i)]
\item $\findim\Lambda\leq \maxx\big\{\findim\Gamma, 1 \big\}$.
\item $\findim\Gamma\leq \maxx\big\{\findim\Lambda, 1 \big\}$.
\end{enumerate} 
In particular, 
\[
\findim\Lambda < \infty \text{\ if and only if\ } \findim\Gamma < \infty.
\]
\end{thm}

\begin{rem}
The arrow removal operation has also been considered by Diracca and Koenig in \cite{DK}. They considered the notion of an exact split pair $(\mi, \me)$, i.e.\ a pair of exact functors $\mi\colon \B\lxr \A$ and $\me\colon \A\lxr \B$ between abelian categories such that the composition $\me\mi$ is an auto-equivalence of $\B$. Up to this auto-equivalence, $\me$ being faithful and the existence of a left adjoint of $\me$, this is a cleft extension as defined in Definition~\ref{defncleftext}. Arrow removal of an arrow $a$ which only occurs in monomial relations gives rise to an exact split pair, see \cite[Proposition 5.4 (b)]{DK}. We now show that the induced cleft extension of such an arrow removal does not in general satisfy the conditions of Theorem~\ref{thm:cleftextfindim}. For example, consider the algebra $\Lambda=
k\left(\xymatrix{1\ar@(ul,dl)_a  }\right)/\langle a^2 \rangle$. The abelian category $\smod\Lambda$ is a cleft extension of $\smod{k}$, in particular, we have the following diagram:
\[
\xymatrix@C=0.5cm{
\smod{k}  \ar@(ul,dl)_{F} \ar[rrr]^{\mi = \Hom_k({_\Lambda k}_k,-)} &&& \smod\Lambda
\ar@(ul,ur)^{G} \ar[rrr]^{\me = \Hom_\Lambda({_k\Lambda}_\Lambda,-)} 
\ar @/_1.5pc/[lll]_{\mq = {-\otimes_\Lambda  {_\Lambda k}_k}}  \ar
 @/^1.5pc/[lll]^{\map = \Hom_\Lambda({_kk}_\Lambda,-)} &&& \smod{k} \ar@(ur,dr)^{F}
\ar @/_1.5pc/[lll]_{\ml = {-\otimes_k {_k\Lambda}_\Lambda}} \ar
 @/^1.5pc/[lll]^{\mr = \Hom_k({_\Lambda\Lambda}_k,-)} }
\]
The right hand side is induced from the natural inclusion $k\lxr \Lambda$. Then, the functor $\ml$ is exact, the functor $\mr$ is exact and therefore the functor $\me$ preserves projectives. Moreover, since $\smod{k}$ is a semisimple category the image of the endofunctor $F$ is projective.
Let us now compute the endofunctor $G$. Let $X$ be a $\Lambda$-module. From the exact sequence \eqref{secondexactsequence} we have the map 
\[
\ml\me(X)=X|_k\otimes_k\Lambda_{\Lambda} \lxr X \lxr 0, \ x\otimes \lambda\mapsto x\lambda
\]
and the endofunctor $G$ is the kernel. Clearly, $X|_k\otimes_k\Lambda_{\Lambda}$ is a finitely generated projective $\Lambda$-module. Thus, $G(X)$ is the first syzygy of $X$ plus some projective. This shows that for all $X$ in $\smod\Lambda$ the first syzygy of $X$ is a direct summand of $G(X)$ and therefore $G(X)$ is not projective. Hence, the second part of condition \eqref{conditionsfortheImageFandG} is not satisfied.
\end{rem}

\section{Vertex removal and finitistic dimension}
\label{sectionvertexremoval}

This section is devoted to discussing reduction techniques for
finitistic dimension in abelian categories with enough projectives
occurring in recollement situations.  This is done in four
situations, two of which have occured in the literature already
(see \cite[Corollary 4.21 1)]{FosGriRei}, \cite[Proposition 2.1]{Ful-Sao:FinDimConjArtRings}) and two are new.  These are applied to finite dimensional algebras.  

\subsection{Triangular reduction}
A form of triangular reduction was considered by Happel in \cite{Happel} via recollements: 

\begin{thm}
\label{thmHappel}
If a finite dimensional algebra $\Lambda$ occur in a recollement of bounded
derived categories like
\[
\xymatrix@C=0.5cm{
\mathsf{D}^{\mathsf{b}}(\smod\Lambda'') \ar[rrr]^{\mi} &&& \mathsf{D}^{\mathsf{b}}(\smod\Lambda) \ar[rrr]^{\me}  \ar @/_1.5pc/[lll]_{\mathsf{\mq}}  \ar
 @/^1.5pc/[lll]^{\mp} &&& \mathsf{D}^{\mathsf{b}}(\smod\Lambda') 
\ar @/_1.5pc/[lll]_{\ml} \ar
 @/^1.5pc/[lll]^{\mr}
 } 
\]
then $\findim\Lambda<\infty$ if and only if $\findim\Lambda''<\infty$ and $\findim\Lambda'<\infty$.  
\end{thm}

When such a recollement exists is characterized in \cite{Koenig}, but it seems hard to apply.  A reduction formula for the finitistic dimension of
triangular matrix rings is given by the following classical result due
to Fossum, Griffith and Reiten. 

\begin{thm}[\protect{\!\!\!\cite[Corollary 4.21)]{FosGriRei}}]
\label{thmFGR}
  Let
  $\Lambda =\left(\begin{smallmatrix} R & 0\\ M &
      S\end{smallmatrix}\right)$ for rings $R$ and $S$ and a non-zero
  $S$-$R$-bimodule $M$.  Then 
\[
\findim R\leq \findim \Lambda \leq 1 + \findim R  + \findim S.
\]
\end{thm}
 
Let $\Lambda$ be as in Theorem~\ref{thmFGR}, and let
$e = \left(\begin{smallmatrix} 1 & 0 \\ 0 & 0\end{smallmatrix}\right)$
an idempotent element in $\Lambda$.  The triangular ring gives rise to
a recollement situation as in Example \ref{examrecolmodcat} with
$(1 - e)\Lambda(1-e) \simeq \Lambda/\Lambda e \Lambda$.  We apply this
reduction technique to finite dimensional quotients $kQ/I$ of path
algebras to show that $Q$ is path connected if no triangular reduction
is possible (first proved in \cite{GreenMarcos}).

Given an idempotent $e$ in a finite dimensional algebra $\Lambda$ we
can view $\Lambda$ as the matrix ring
\[\left( \begin{matrix} e\Lambda e & e\Lambda (1 - e) \\
      (1 -e)\Lambda e & (1 - e)\Lambda(1 - e)\end{matrix}\right).\] If
$e\Lambda (1 - e)$ (or $(1 - e)\Lambda e$) equals zero for any
idempotent $e\neq 0, 1$, then we say that $\Lambda$ has a
\emph{triangular structure}. Furthermore, an algebra $\Lambda$ is
  said to be \emph{triangular reduced} if $\Lambda$ has no non-trivial
  triangular structures.  We use Theorem~\ref{thmFGR} to reprove the following result from \cite{GreenMarcos}: 
If a finite dimensional algebra $\Lambda = kQ/I$ has no non-trivial
triangular structure, then the quiver $Q$ is path connected.

Let $\Lambda = kQ/I$ be an admissible quotient of the path algebra
$kQ$ over a field $k$, and let $e$ be a sum of vertices in $Q$ with
$e\neq 0,1$.  Assume that $e\Lambda (1 - e) =(0)$.  If $(1 - e)\Lambda
e = (0)$ instead, interchange the role of $e$ and $1 - e$.  Then 
\[\Lambda\simeq \left(\begin{smallmatrix} e\Lambda e & 0\\
      (1 - e)\Lambda e & (1 - e)\Lambda
      (1-e)\end{smallmatrix}\right).\] 
Then $\findim \Lambda$ is finite if $\findim e\Lambda e$ and
$\findim(1 - e)\Lambda (1 - e)$ are finite. Next, a triangular reduced
admissible quotient $\Lambda = kQ/I$ of a path algebra is shown to
have a path connected quiver $Q$.

\begin{prop}[\protect{\!\cite{GreenMarcos}}]
Let $\Lambda = kQ/I$ be an admissible quotient of the path algebra
$kQ$ over a field $k$. Assume that $\Lambda$ is triangular reduced. Then $Q$ is path connected.
\end{prop}
\begin{proof}
  Let $Q_0 = \{v_1,v_2,\ldots,v_n\}$.  Let $e = v_1$. Then
  $v_1\Lambda( \sum_{i=2}^n v_i) \neq (0)$. Hence there is an arrow
  $v_1\lxr v_{i_2}$ for $v_{i_2}\neq v_1$.  Let $i_1 = 1$.  Now let
  $e = v_{i_1} + v_{i_2}$.  Again using that $e\Lambda(1-e) \neq (0)$,
  we infer that there is an arrow $v_{i_1}\lxr v_{i_3}$ or
  $v_{i_2}\lxr v_{i_3}$ for some $v_{i_3}\not\in \{v_{i_1},v_{i_2}\}$.
  We can continue this process and get subsets
  $\V_t = \{v_{i_1}, v_{i_2},\ldots, v_{i_t}\}$ in $Q_0$ and a vertex
  $v_{i_{t+1}}\not\in \V_t$ with an arrow from some
  $v_{i_{j}}\lxr v_{i_{t+1}}$ for $j\in\{ 1,2,\ldots, t\}$.  We can only
  continue this construction until we reach $\V_n$, since then
  $e = \sum_{j=1}^n v_{i_j} = 1$ and $1 - e = 0$.  In other words,
  there is a path from the vertex $v_1$ to any other vertex different
  from $v_1$ in $Q$. Since $v_1$ can be chosen to be any vertex in
  $Q$, the quiver $Q$ is path connected.
\end{proof}

A concept of a homological heart of a quotient of a path algebra is
introduced in \cite{GreenMarcos}, which we recall and discuss next.
As above let $\Lambda = kQ/I$ be an admissible quotient of the path
algebra $kQ$ over a field $k$.  The homological heart of $Q$ is
given as follows:  Let 
\[X = \{v\in Q_0\mid v \textrm{\ is a vertex on a non-trivial
      oriented cycle in $Q$}\}\]
and let 
\[Y = \{v\in Q_0\mid y \textrm{\ is a vertex on a path starting and
    ending in $X$}.\]
Then the \emph{homological heart $H(Q)$ of $Q$} is the
subquiver of $Q$ with vertex set $Y$.  

In order to discuss the properties of the homological heart of a
quiver, we need to introduce the following.  Let $\Gamma$ be a full
subquiver of $Q$.  Define the following three full subquivers of $Q$
by their corresponding vertex sets:
\[\Gamma^+_0 = \{v\in Q_0\mid v\not\in \Gamma_0, \exists \textrm{\
    path\ } \Gamma_0 \leadsto v\}\]
\[\Gamma^-_0 = \{v\in Q_0\mid v\not\in \Gamma_0, \exists \textrm{\
    path\ } v \leadsto \Gamma_0 \}\]
\[\Gamma^o_0 = \{v\in Q_0\mid v\not\in \Gamma_0, \not\exists \textrm{\
    path\ } \Gamma_0 \leadsto v \textrm{\ and\ } \not\exists \textrm{\
    path\ }  v \leadsto \Gamma_0 \}\]
Let $e^+(\Gamma)$, $e^-(\Gamma)$ and $e^o(\Gamma)$ be the sum of all
the vertices in $\Gamma^+$, $\Gamma^-$ and $\Gamma^o$, respectively. 

The homological heart $H=H(Q)$ is a full subquiver.  Let $e^+$, $e^-$
and $e^o$ the corresponding idempotents defined above for $\Gamma = H$
and $e$ the sum of the vertices in $H$.  Then the following is proved in
\cite[Theorem 5.9 (2)]{GreenMarcos}.
\begin{thm}[\protect{\!\!\cite[Theorem 5.9 (2)]{GreenMarcos}}]
Let $\Lambda = kQ/I$ be an admissible quotient of the path
algebra $kQ$ over a field $k$ with homological heart $H=H(Q)$ and $e$
the sum of the vertices in $H$.  Then $\findim \Lambda < \infty$  if
and only if $\findim e\Lambda e < \infty$. 
\end{thm}
We want to show that this result can be obtained from the reduction
techniques presented in this paper.  Let 
\[1_\Lambda = e^+ + e^o + e + e^-,\] 
then 
\[\Lambda \simeq \begin{pmatrix}
e^+\Lambda e^+ & e^+\Lambda e^o & e^+\Lambda e & e^+\Lambda e^-\\
e^o\Lambda e^+ & e^o\Lambda e^o & e^o\Lambda e & e^o\Lambda e^-\\
e\Lambda e^+ & e\Lambda e^o & e\Lambda e & e\Lambda e^-\\
e^-\Lambda e^+ & e^-\Lambda e^o & e^-\Lambda e & e^-\Lambda e^-
\end{pmatrix}\]
Here 
\[e^+\Lambda e = e^+\Lambda e^- = e^o\Lambda e = e^o\Lambda e^- = e^+\Lambda e^o = e\Lambda e^- = e\Lambda e^o = (0).\]
Hence 
\[\Lambda \simeq \begin{pmatrix}
e^+\Lambda e^+ & 0 & 0 & 0\\
e^o\Lambda e^+ & e^o\Lambda e^o & 0 & 0\\
e\Lambda e^+ & 0 & e\Lambda e & 0\\
e^-\Lambda e^+ & e^-\Lambda e^o & e^-\Lambda e & e^-\Lambda e^-
\end{pmatrix}\]
By \cite[Proposition 5.1 (4)]{GreenMarcos} the full subquiver with
vertex set $H^+_0\cup H^-_0\cup H^o_0$ has no oriented cycles, therefore the
algebras $\left(\begin{smallmatrix} 
e^+\Lambda e^+ & 0 \\
e^o\Lambda e^+ & e^o\Lambda e^o 
\end{smallmatrix}\right)$ and $e^-\Lambda e^-$ have finite global
dimension.  Iterated use of Theorem \ref{thmFGR} show that if $\findim
e\Lambda e < \infty$, then $\findim\Lambda < \infty$.    

Assume now conversely that $\findim\Lambda < \infty$ and consider the idempotent element of $\Lambda\colon$
\[
f=\begin{pmatrix}
0 & 0 & 0 & 0\\
0 & 0 & 0 & 0\\
0 & 0 & 0 & 0\\
0 & 0 & 0 & 1
\end{pmatrix}.
\]
Then by Example~\ref{examrecolmodcat} we have the following recollement of module categories:
\[
\xymatrix@C=0.5cm{
\smod{A} \ar[rrr]^{\mi} &&& \smod{\Lambda} \ar[rrr]^{\mathsf{f}:=f(-) \ \ } \ar
@/_1.5pc/[lll]_{}  \ar
 @/^1.5pc/[lll]^{} &&& \smod{e^-\Lambda e^-}
\ar @/_1.5pc/[lll]_{} \ar
 @/^1.5pc/[lll]^{}
 } 
\]
where
\[
A=\begin{pmatrix}
e^+\Lambda e^+ & 0 & 0 \\
e^o\Lambda e^+ & e^o\Lambda e^o & 0 \\
e\Lambda e^+ & 0 & e\Lambda e 
\end{pmatrix}.
\]
Since $\Lambda$ is a triangular matrix algebra, it follows by \cite[Theorem 3.9]{Psaroud} that the functor $\mi$ is a homological embedding. Since $\gld{e^-\Lambda e^-}<\infty$, by \cite[Theorem~7.2 (ii)]{Psaroud} we get a lifting of the above recollement to a recollement situation at the level of bounded derived categories as follows:
\[
\xymatrix@C=0.5cm{
\mathsf{D}^{\mathsf{b}}(\smod{A}) \ar[rrr]^{\mi} &&& \mathsf{D}^{\mathsf{b}}(\smod\Lambda) \ar[rrr]^{\mathsf{f} \ \ }  \ar @/_1.5pc/[lll]_{}  \ar @/^1.5pc/[lll]^{} &&& \mathsf{D}^{\mathsf{b}}(\smod e^-\Lambda e^-) 
\ar @/_1.5pc/[lll]_{} \ar @/^1.5pc/[lll]^{}
 } 
\]
By Theorem~\ref{thmHappel}, it follows that $\findim{A}<\infty$. Then,
as above, the module category of $A$ admits the following recollement
situation:
\[
\xymatrix@C=0.5cm{
\smod{A/Af'A} \ar[rrr]^{} &&& \smod{A} \ar[rrr]^{ f'(-)} \ar
@/_1.5pc/[lll]_{}  \ar
 @/^1.5pc/[lll]^{} &&& \smod{f'Af'}
\ar @/_1.5pc/[lll]_{} \ar
 @/^1.5pc/[lll]^{}
 } 
\]
where $f'=\bigl(\begin{smallmatrix}
0 & 0 & 0 \\
0 & 0 & 0 \\
0 & 0 & 1
\end{smallmatrix}\bigr)$, $A/Af'A\simeq \bigl(\begin{smallmatrix}
e^+\Lambda e^+ & 0 \\
e^o\Lambda e^+ & e^o\Lambda e^o
\end{smallmatrix}\bigr)$ and $f'Af'\simeq e\Lambda e$. The left
adjoint $\ml$ of $f'(-)$ is an exact functor. Then from
\cite[Theorem~5.1]{Psaroud}, or by just using that $(\ml, f'(-))$ is
an adjoint pair and $\ml$ is exact which preserves projective modules,
we get that $\findim{e\Lambda e}\leq \findim{A}$. We infer that
$\findim{e\Lambda e}<\infty$.

\subsection{Vertex removal, projective dimension at most $1$} 
Let $(\A,\B,\C)$ be a recollement of abelian categories.  Recall that
\[
{\pgld}_{\B}\A = \{\pd_\B \mi(A)\mid A\in \A\}
\]
denotes the $\A$-relative projective global dimension of $\B$.  If
$\pgld_{\B}\A \leq 1$ we show that $\Findim\B$ is finite if and only if
$\Findim\C$ is finite.  Since $\C\simeq \B/\A$, we can interpret the
result as follows: We can remove $\A$ from $\B$ and not lose any
information about the finiteness of the finitistic dimension.  For an
algebra $\Lambda$ it means that $\Lambda$ and $e\Lambda e$ has
mutually finite finitistic dimensions for an idempotent $e$ whenever
$\pd_\Lambda (1-e)\Lambda/(1-e)\mathfrak{r} \leq 1$.  If $\Lambda$ is
a quotient of a path algebra and $e$ is a sum of vertices, then the
vertices in the quiver of $e\Lambda e$ correspond to the ones
occurring in $e$, that is, the vertices occurring in $1-e$ are
removed.  This is why we call the transition from $\Lambda$ to
$e\Lambda e$ \emph{vertex removal}.  

We start with the general situation of a recollement of abelian
categoies.

\begin{thm}
\label{thmfindimpgld}
Let $(\A,\B,\C)$ be a recollement of abelian categories. Assume that
$\pgld_{\B}{\A}\leq 1$. Then the following hold: 
\[
\maxx\{\Findim{\A}, \Findim{\C}\} \leq \Findim{\B} \leq \Findim{\C}+2.
\] 
In particular, we have that $\Findim{\B}<\infty$ if and only if
$\Findim{\C}<\infty$. 
\end{thm}
\begin{proof}
Since $\pgld_{\B}{\A}\leq 1$, it follows from
Proposition~\ref{recollfiniteArelglodimofB} (iv) 
that the quotient functor $\me\colon \B\lxr \C$ preserves projective
objects. Then from \cite[Theorem 5.5]{Psaroud} (ii) it follows that
$\Findim{\B} \leq \Findim{\C}+2$. Also from \cite[Theorem 5.5]{Psaroud}
(i) we have $\Findim{\A} \leq \Findim{\B}$. It remains to show that
$\Findim{\C} \leq \Findim{\B}$.

  Let $C$ be an object in $\C$ of finite projective dimension. Suppose
  that $\Findim{\B}=m<\infty$. By the formula
  $\pd_{\B}\ml(C)\leq \pd_{\C}C+\pgld_{\B}{\A}+1$, which is the dual
  of formula \eqref{formulainjectdim} in the proof of
  Lemma~\ref{lempigldfinite}, we have that
  $\pd_\B\ml(C) \leq \pd_{\C}C+2$ and therefore $\pd_{\B}\ml(C)\leq
  m$. Thus we have an exact sequence
\[
 \xymatrix{
  0 \ar[r]^{ } & P_m \ar[r]^{} & \cdots \ar[r]^{} &
  P_0  \ar[r]^{} & \ml(C) \ar[r] & 0  }
\]
with $P_i$ in $\Proj{\B}$. Applying the functor $\me$ and using 
Proposition~\ref{properties} (iv) we get the exact sequence
\[
 \xymatrix{
  0 \ar[r]^{ } & \me(P_m) \ar[r]^{} & \cdots \ar[r]^{} &
  \me(P_0)  \ar[r]^{} & C \ar[r] & 0  }
\]
where each $\me(P_i)$ lies in $\Proj{\C}$. We infer that
$\pd_{\C}C\leq m = \Findim{\B}$ and this completes the proof.
\end{proof}

\begin{rem}
  The inequality of Theorem~\ref{thmfindimpgld} was proved in
  \cite[Proposition 4.15]{Psaroud} to hold for the global
  dimension. Note also that the upper bound of
  Theorem~\ref{thmfindimpgld} generalizes \cite[Proposition
  2.1]{Ful-Sao:FinDimConjArtRings} from basic left Artin rings to the
  setting of abstract abelian categories, thus to any recollement of
  module categories with $R/ReR$-relative projective global dimension
  ${\pgld}_{R/ReR}{R}\leq 1$.
\end{rem}

Let $\Lambda$ be an artin algebra with an idempotent $e$.  As in
  Example \ref{examrecolmodcat} it gives rise to the recollement
  $(\smod\Lambda/\Lambda e\Lambda, \smod\Lambda, \smod e\Lambda e)$,
  where $\Lambda/\Lambda e\Lambda$ is also an artin algebra and
  therefore all modules are filtered in semisimple modules.  Then
  $(1 - e)\Lambda/(1 - e)\mathfrak{r}$ is right
  $\Lambda/\Lambda e\Lambda$-module.  In fact it is semisimple, where
  all simple $\Lambda/\Lambda e\Lambda$-modules occur as a direct
  summand.  Then
  $\pgld_{\smod\Lambda/\Lambda e\Lambda}\smod\Lambda \leq 1$ if and
  only if $\pd_\Lambda (1 - e)\Lambda/(1 - e)\mathfrak{r} \leq 1$.
  Using this we have the following immediate consequence of the above.

\begin{cor}\label{cor:recollfpd-projdim}
  Let $\Lambda$ be an artin algebra with an idempotent $e$.  Assume
  that $\pd_\Lambda  (1-e)\Lambda/ (1-e)\rad(\Lambda) \leqslant 1$.  Then
  $\findim\Lambda<\infty$ if and only if $\findim e\Lambda e<\infty$.
\end{cor}

\subsection{Vertex removal, finite injective dimension}
Let $(\A,\B,\C)$ be a recollement of abelian categories with
  enough injectives and projectives. Recall that 
\[
{\igld}_{\B}\A = \{ {\id}_{\B}\mi(A)\mid A\in \A\}
\]
denotes the $\A$-relative injective global dimension of $\B$.  If
$\igld_\B\A\leq 1$, the same statement as in Theorem \ref{thmfindimpgld} is
true, namely $\Findim\B$ is finite if and only if $\Findim\C$ is finite, see Remark~\ref{remlowerbound} below. This has the same translation for an artin algebra $\Lambda$ as in Corollary \ref{cor:recollfpd-projdim}, namely through the condition $\id_\Lambda(1-e)\Lambda/(1-e)\mathfrak{r} < \infty$.

Next we prove the main result of this subsection
which gives an upper bound for the finitistic dimension of $\B$ using
the finiteness of the $\A$-relative injective global dimension of
$\B$. This is a general version of Theorem~B presented in the Introduction.

\begin{thm}
\label{thmigldfinite}
Let $(\A,\B,\C)$ be a recollement of abelian categories such that $\A$ is a finite length category and that $\B$ has projective covers. 
Then we have
\[
\Findim{\B}\leq \Findim{\C}+{\igld}_{\B}{\A}
\]
\end{thm}
\begin{proof}
Assume that $\igld_{\B}{\A}=\sup \{{\id}_{\B}\mi(A) \ | \ A\in\A\}=t<\infty$. Let $X$ be an object in $\B$ of finite projective dimension. Then from Lemma~\ref{lemtruncatedprojresol} there is a finite projective resolution as follows:
\begin{equation}
\label{finitetruncatedresolution}
\xymatrix{
\cdots \ar[r] &
\ml(Q_1) \ar[r] &
\ml(Q_0) \ar[r] &
P_{t} \ar[r] &
\cdots \ar[r] &
P_0 \ar[r] &
X \ar[r] &
0
} 
\end{equation}
with $Q_1$ and $Q_0$ in $\Proj\C$. Applying the functor
$\me\colon \B\lxr \C$ to $(\ref{finitetruncatedresolution})$ and using
Proposition~\ref{properties} (iv), we obtain the exact sequence
\[
\xymatrix{
\cdots \ar[r] &
Q_1 \ar[r] &
Q_0 \ar[r] & e(\Omega^t(X)) \ar[r] & 0
}
\]
which is a finite projective resolution of $\me(\Omega^t(X))$. This
implies that $\pd_{\C}\me(\Omega^t(X))\leq \Findim{\C}$. Hence, the
length of the resolution $(\ref{finitetruncatedresolution})$ is
bounded by $\Findim{\C} + t$. We conclude that $\pd{_{\B}X}\leq
\Findim{\C}+t$, and the claim follows.
\end{proof}

\begin{rem}
\label{remlowerbound}
It is natural to ask if in Theorem~\ref{thmigldfinite} we can get a
lower bound for $\Findim{\B}$. We show that this is indeed the case if
we assume that ${\igld}_{\B}{\A}\leq 1$. Assume that
$\Findim\B=m<\infty$ and let $C$ be an object of $\C$ of finite
projective dimension. Since the injective relative dimension
${\igld}_{\B}{\A}\leq 1$, it follows from the dual of
Proposition~\ref{recollfiniteArelglodimofB} that the functor
$\mi\colon \A\lxr \B$ is a homological embedding and the functor
$\me\colon \B\lxr \C$ preserves injective objects. The latter
condition is equivalent to the functor $\ml\colon \C\lxr \B$ being
exact. Since $\ml$ is exact and preserves projectives, we get that
$\pd_{\C}C=\pd_{\B}\ml(C)\leq \Findim\B=m<\infty$. This implies that
$\Findim\C\leq \Findim\B$ and so we are done.
\end{rem}

We close this subsection with the following consequence of
Theorem~\ref{thmigldfinite} for Artin algebras.

\begin{cor}\textnormal{(Solberg \cite{Solberg})}\label{cor:Solberg}
  Let $\Lambda$ be an Artin algebra and $e$ an idempotent
  element. Then
\[
  \findim\Lambda\leq \findim e\Lambda e+\mathsf{sup}\{\id{_{\Lambda}S}  \ | \ S \ \text{simple} \ \Lambda/\Lambda e\Lambda\text{-module}\}
\] 
\end{cor}

\subsection{Vertex removal, eventually homological isomorphism}
Let $(\A,\B,\C)$ be a recollement of abelian categories with enough
injectives and projectives. Recall that the functor
$\me\colon \B\lxr \C$ is a $t$-eventually homological isomorphism for
some $t$ if $(\alpha)$ ${\igld}_{\B}{\A}<\infty$ and $(\beta)$
$\sup \{ \pd_\C \me(P) \mid P \in \Proj \B \} < \infty$, see
Proposition~\ref{prop:algebra-ext-iso}.  We see that this includes the
condition ${\igld}_{\B}{\A}<\infty$ from the previous subsection.
Hence if $\me\colon \B\lxr \C$ is a $t$-eventually homological
isomorphism, there should be an even closer relationship between
finitistic dimension of $\B$ and $\C$.  This is shown in the next
result.

\begin{thm}
\label{thmevenfindim}
Let $(\A, \B, \C)$ be a recollement of abelian categories such that
$\A$ is a finite length category and $\B$ has projective
covers. Assume that the functor $\me\colon \B\lxr \C$ is a
$t$-eventually homological isomorphism. Then the following statements
hold.
\begin{enumerate}[\rm(i)]
\item $\Findim \C \leq \maxx\{\Findim \B, t\}$.
\item $\Findim \B \leq \maxx\{\Findim \C, t\}$. 
\end{enumerate}
In particular,  $\Findim\B$ is
  finite if and only if $\Findim \C$ is finite.
\end{thm}
\begin{proof}
  By assumption there is an isomorphism
  $\Ext^n_{\B}(B,B')\simeq \Ext^n_{\C}(\me(B),\me(B'))$ for every $n>t$
  and for all objects $B$, $B'$ in $\B$.

  (i) Assume that the finitistic dimension of $\B$ is finite. Let $Y$
  an object in $\C$ of finite projective dimension. By Proposition
  \ref{properties} (iv) and Proposition~\ref{prop:algebra-ext-iso}, we
  have
\begin{equation}
\label{isomeventhomol}
{\Ext}^n_{\B}(\ml(Y), B')\simeq {\Ext}^n_{\C}(Y,\me(B'))
\end{equation}
for all integers $n > t$ and for all objects $B'$ in $\B$.  Since
the functor $\me$ is essentially surjective by Proposition
\ref{properties} (iii), it follows from $(\ref{isomeventhomol})$ that
$\pd_\B \ml(Y) < \infty$ and in particular
$\pd_\B\ml(Y) \leq \Findim\B$. From the isomorphism
\eqref{isomeventhomol} it follows that
\[
\pd_{\C}Y\leq \maxx\{\Findim\B,t \}.
\]
We infer that $\Findim\B \leq \maxx\{\Findim \C, t\}$.

(ii) Assume that the finitistic dimension of $\C$ is
finite. Let $B$ be an object in $\B$ of finite projective
dimension. Then for any object $B'$ in $\B$ we have the 
isomorphism
${\Ext}^n_{\B}(B, B')\simeq {\Ext}^n_{\C}(\me(B),\me(B'))$ for
every $n>t$. As above, we obtain that
$\pd_{\C}\me(B)\leq \Findim \C$ and therefore
\[
\pd_{\B}B\leq \maxx\{\Findim \C, t \}
\]
Hence, we conclude that $\Findim\B \leq \max\{\Findim \C,
t\}$. 

The last claim follows directly from (i) and (ii). 
\end{proof}

As a consequence of Theorem~\ref{thmevenfindim} we have the following
result for Artin algebras. 

\begin{cor}
\label{corehifindim}
Let $\Lambda$ be an artin algebra with an idempotent $e$.  Assume
  that the functor $e\colon \smod\Lambda\lxr \smod e\Lambda e$ is a
  $t$-eventually homological isomorphism. Then the following statements hold. 
\begin{enumerate}[\rm(i)]
\item $\findim e\Lambda e \leq \maxx\{\findim\Lambda, t\}$.
\item $\findim \Lambda \leq \maxx\{\findim e\Lambda e, t\}$. 
\end{enumerate}
In particular,  $\findim\Lambda$ is
  finite if and only if $\findim e\Lambda e$ is finite.
\end{cor}

Using Corollary~\ref{corehifindim} we can reprove
Corollary~\ref{cor:recollfpd-projdim} in the following way. Assume
that $\pd_\Lambda (1-e)\Lambda/ (1-e)\rad(\Lambda) \leqslant 1$. By
Proposition~\ref{recollfiniteArelglodimofB} the inclusion functor
$\inc\colon \smod\Lambda/\Lambda e\Lambda\lxr \smod \Lambda$ is a
homological embedding and the quotient functor
$e\colon \smod\Lambda\lxr \smod e\Lambda e$ preserves
projectives. Hence, from Proposition~\ref{prop:algebra-ext-iso} we
have the properties $(\beta)$ and $(\gamma)$ but in general this is
not enough to obtain an eventually homological isomorphism. However,
it follows from Lemma~$8.9$ and Corollary~$8.8$ in \cite{PSS} that the
functor $e$ is an eventually homological isomorphism and therefore we
can apply Corollary~\ref{corehifindim}.

\begin{rem}
  All the reductions that we have discussed give at least one thing,
  depending on your point of view.  If you believe the finitistic
  dimension conjecture is true, then you ``only'' need to prove it for
  $\Lambda = kQ/I$ where $Q$ is path connected, all simple modules
  have infinite injective dimension and projective dimension at least
  $2$, and all arrows occur in a given minimal generating set for the
  relations.  If you believe the finitistic dimension conjecture is
  false, then you ``only'' need to search for/find a counter example
  with all the properties just given.
\end{rem}

\section{Examples}
\label{sectionexamples}

This section is devoted to giving examples illustrating the reduction
techniques we have discussed in the previous sections.  In this context we introduce the following notion.

\begin{defn}
\label{defn:reduced}
An algebra in called {\bf reduced} if 
\begin{enumerate}[\rm(a)]
\item all arrows occur in some relation in a minimal set of relations,  
\item all simple modules have infinite injective dimension and
projective dimension at least $2$, 
\item no triangular reductions are possible.
\end{enumerate}
\end{defn}

The first two examples are not possible to reduce using the earlier
known techniques with vertex reduction for vertices corresponding to
simple modules of projective dimension at most $1$ or triangular
reduction.  In these examples one must use the new techniques with
vertex reduction for vertices corresponding to simple modules of
finite injective dimension or arrow removal, respectively.  Reduced
algebras $\Lambda_n$ for all positive integers $n$ are constructed
with the finitistic dimension equal to $n$.  This shows that the
finitistic dimension can be arbitrary for a reduced algebra.  Then,
all the rest of the examples are from the existing literature, and
they are shown to be reducible using our techniques.  In most cases a
bound for the finitistic dimension of the algebra is given.

In all the examples the algebras $\Lambda$ are given by quivers and
relations.  Then denote by $S_i$ the simple module associated to
vertex number $i$ and by $e_i$ the corresponding primitive idempotent.
Furthermore, $P_i$ denotes the indecomposable projective modules
$e_i\Lambda$.

First we give two examples which can only be reduced by applying the
new reduction techniques introduced in this paper, namely vertex
removal corresponding to a simple module with finite injective
dimension and arrow removal.

\begin{exam}
  Let $\Lambda$ be given as $k\Gamma/I$ for a field $k$, where
  $\Gamma$ is the quiver
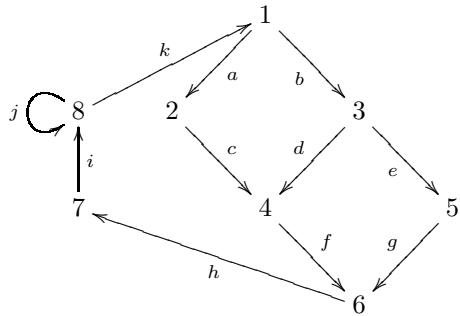
\[\xymatrix{
&  & 1\ar[dl]^a\ar[dr]_b & & \\
8\ar[urr]^k\ar@(ul,dl)_j &  2\ar[dr]^c & & 3\ar[dl]_d\ar[dr]_e & \\
7\ar[u]_i&  & 4\ar[dr]^f & & 5\ar[dl]_g \\
&                &                 & 6\ar[ulll]^h &  
}\]
and
$I=\langle ac-bd, be, cf, df-eg, fh, egh, ghi, hik, ij, ikb, j^2, jkb,
ka, kbd\rangle$.  One can show that (i) all simple $\Lambda$-modules
have infinite injective dimension, except $S_7$, which has injective
dimension $3$, (ii) all simple $\Lambda$-modules have (here in fact
infinite) projective dimension at least $2$, (iii) no triangular
reductions are possible and (iv) no arrow removal is possible.
Therefore the only available reduction is vertex removal corresponding
to vertex $7$, where the simple module has injective dimension $3$. 

The finitistic dimension is $\Lambda$ is $2$, while the finitistic
dimension of $\Lambda^\op$ is $4$.
\end{exam}

In the next example only arrow removal is possible. 
\begin{exam}
Let $\Lambda$ be given as $k\Gamma/I$ for a field $k$, where $\Gamma$
is 
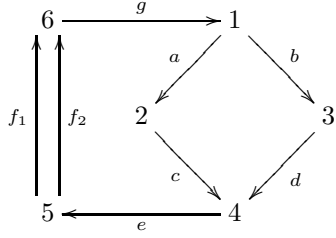
\[\xymatrix{
6\ar[rr]^g && 1\ar[dl]_a\ar[dr]^b & \\
                & 2\ar[dr]_c & & 3\ar[dl]^d\\
5\ar@<1ex>[uu]^{f_1}\ar@<-1ex>[uu]_{f_2} & & 4\ar[ll]^e & }\]
and $I = \langle ac - bd, cef_1, de, ef_1g, f_1gb, ga\rangle$.  Then
all the simple modules are either $\Omega$-periodic
($\Omega^{-1}$-periodic) or eventually $\Omega$-periodic
($\Omega^{-1}$-periodic) of periode $11$.  Hence all the simple
modules have infinite injective and infinite projective dimension and
the algebra is reduced with respect to vertex removal.  Furthermore,
the algebra is triangular reduced.  Therefore the only possible
reduction we can perform is arrow removal as $f_2$ is not occurring in
any relations.

By Theorem \ref{thm:summarycleftext} the inequality
$\findim\Lambda \leq \max\{\findim\Lambda/\langle f_2\rangle, 1\}$ is true.
Since $\findim\Lambda/\langle f_2\rangle = 1$, it follows that
$\findim\Lambda \leq 1$.
\end{exam}

Next we construct algebras $\Lambda_n$ for all positive integers $n$,
which are reduced with $\findim \Lambda_n = n$.  Hence, a reduced
algebra can have arbitrarily large finitistic dimension. 

\begin{exam}
\label{exam:Koszul}
Let $\Lambda_n$ be given as $k\Gamma_n/I_n$, where $\Gamma_n$ is given
by
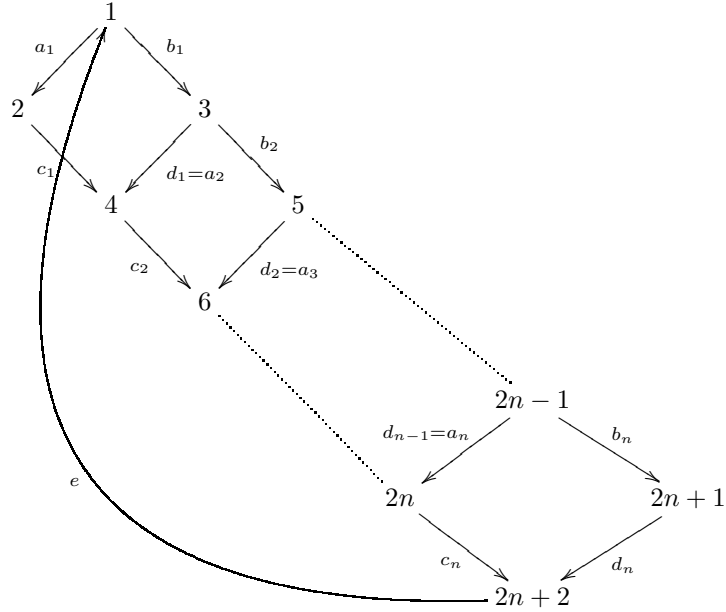
\[\xymatrix{
 && 1\ar[dl]_{a_1}\ar[dr]^{b_1} & & & & & \\
                & 2\ar[dr]_{c_1} & & 3\ar[dl]^{d_1=a_2}\ar[dr]^{b_2} & & & &\\
 & & 4\ar[dr]_{c_2} & & 5\ar[dl]^{d_2=a_3}\ar@{..}[ddrr] & & & \\
 & & & 6 \ar@{..}[ddrr] & & & & \\
 & & & & & & 2n-1\ar[dl]_{d_{n-1}=a_{n}}\ar[dr]^{b_{n}} & \\
 & & & & & 2n\ar[dr]_{c_{n}} & & 2n+1\ar[dl]^{d_{n}}\\
 & & & & & & 2n+2\ar@/^9pc/[uuuuuullll]^e &
}\]
and the ideal $I_n$ in $k\Gamma_n$ is generated by
$\{a_ic_i - b_id_i\}_{i=1}^n$, $\{b_ib_{i+1}\}_{i=1}^{n-1}$,
$\{c_ic_{i+1}\}_{i=1}^{n-1}$ and $\{c_ne, d_ne, ea_1, eb_1\}$.   All
the algebras $\Lambda_n$ are reduced, Koszul and of finite representation
type.  Since $\fpd\Lambda_n = n $, a reduced algebra can
have arbitrarily large finitistic dimension.
\end{exam}

All the remaining examples are taken from the existing literature.  We
shall see that all of them can be reduced using one of our reduction
techniques. 
\begin{exam}\label{exam:example2}
This example is Example 1 from \cite{Xi:FinDimConjI}. Define $\Lambda$
by the following quiver and relations over a field $k$.
\[\Lambda = k\left(\vcenter{\xymatrix{
   & 1\ar@<1ex>[dl]^\beta\ar@<-1ex>[dl]_\gamma\ar[dr]^\eta & & \\
2\ar@(ul,dl)_\alpha \ar[dr]_\delta & & 3\ar[dl]^\xi & 5\ar[l]\\
& 4 & & }}\right)/\langle \alpha^3, \alpha\delta, \beta\delta, \eta\xi - \gamma\delta\rangle\]
The injective and the projective dimensions of the
simple $\Lambda$-modules are given as follows.
\begin{center}
\begin{tabular}{|c|c|c|c|c|c|}\hline
 & $S_1$ & $S_2$ & $S_3$ & $S_4$ & $S_5$ \\ \hline
$\pd$ & 2 & $\infty$ & 1 & 0 & 1\\ \hline
$\id$ & 0& $\infty$ & 1 & $\infty$ & 0\\\hline
\end{tabular}
\end{center}
Using that the projective dimension of the simple modules $\{ S_3,
S_4, S_5\}$ are at most $1$, we can remove the vertices $\{3,4,5\}$
and obtain the algebra 
\[\Lambda_1 = k\left(\xymatrix{2\ar@(ul,dl)_\alpha &
 1\ar@<1ex>[l]^\beta\ar@<-1ex>[l]_\gamma }\right)/\langle\alpha^3\rangle.\]
Which can be further reduced to the local algebra $e_2\Lambda_1 e_2$,
since $S_1$ is injective.

Alternatively, using that the injective dimension of the simple
modules $\{ S_1, S_3, S_5\}$ are at most $1$ in the original algebra,
the vertices $\{1,3,5\}$ can be removed and we obtain the algebra
\[\Lambda_2 =
k\left(\xymatrix{2\ar@(ul,dl)_\alpha\ar[r]^\delta  &
    4}\right)/\langle\alpha^3, \alpha\delta\rangle.\] 
Which can be further reduced to the local algebra $e_2\Lambda_2 e_2$,
since $S_4$ is projective.  Independent of which reduction we carry
out, the original algebra has indeed finite finitistic dimension.  In
addition, whatever order we carry out the reductions, we end up with
the same algebra up to isomorphism.

Using Theorem \ref{thmfindimpgld} and Theorem \ref{thmigldfinite} we
obtain that $\fpd\Lambda \leq 3$.

In addition, applying triangular reduction with $e = e_3 + e_4 + e_5$
we obtain 
\[\fpd \Lambda \leq 1 + \fpd e\Lambda e + \fpd(1 - e)\Lambda(1
  -e)\]
where $(1 - e)\Lambda (1 - e) = \Lambda_1$ and $e\Lambda e$ is
hereditary.  We can again reduce $\Lambda_1$ by triangular reduction
and get $\fpd \Lambda_1 \leq 1 + \fpd k  + \fpd k[x]/\langle
x^3\rangle = 1$.  Collecting these observations we obtain, as above,
$\fpd\Lambda \leq 3$. 

Example $2$ from \cite{Xi:FinDimConjII} is very similar, and it can be
reduced in a similar fashion.  Example $2$ from
\cite{Xi:FinDimConjIII} is different, but it admits similar
reductions. 
\end{exam}

\begin{exam}
This is Example 1 from \cite{Xi:FinDimConjII}.   Define $\Lambda$ by
the following quiver and relations over a field $k$. 
\[\Lambda = k\left(\vcenter{\xymatrix{
 & 3\ar@<-1ex>[d]_\beta & \\
2\ar@(ul,dl)_\tau &
1\ar@<-1ex>[l]_\alpha\ar@<1ex>[l]^\xi\ar@<-1ex>[u]_\delta
\ar@<-1ex>[r]_\psi \ar@<-1ex>[d]_\eta &
4\ar@<-1ex>[l]_\varphi\ar@(ur,dr)^\epsilon \\
& 5\ar@<-1ex>[u]_\gamma & 
}}\right)/
\big\langle 
\begin{matrix}
\delta\beta - \eta\gamma, \varphi\delta, \varphi\eta,
\varphi\psi, \epsilon^2, \psi\epsilon, \varphi\alpha, \varphi\xi,\\
\beta\eta, \beta\psi, \beta\delta, \gamma\delta, \gamma\psi,
\gamma\eta, \tau^2, \alpha\tau, \xi\tau
\end{matrix}\big\rangle
\]
Here all simple modules have infinite injective dimension and
projective dimension at least $2$, and all arrows occur in some
generator of a minimal set of generators for the relations.  But
$e_2\Lambda (1 - e_2) = (0)$, so that we can perform a triangular
reduction and obtain
\[\fpd \Lambda \leq 1 + \fpd e_2\Lambda e_2 + \fpd (1 - e_2)\Lambda (1
  - e_2) = 1 + \fpd (1 - e_2)\Lambda (1 - e_2),\]
since $e_2\Lambda e_2$ is a local algebra. The algebra
$(1 - e_2)\Lambda(1 - e_2)$ is reduced. 
\end{exam}

\begin{exam}
This example is Example 4.4 from
\cite{Ful-Sao:FinDimConjArtRings}. Define $\Lambda$ 
by the following quiver and relations over a field $k$.
\[\Lambda = k\left(\vcenter{\xymatrix{
1\ar@/^/[rr]^\alpha\ar[dr]_\beta &  & 2\ar@/^/[ll]^\gamma\ar@/_/[dl]_(0.7)\delta\\
& 3\ar@/_/[ur]_\epsilon & }}\right)/\langle \alpha\delta, \alpha\gamma,
\beta\epsilon\gamma, \beta\epsilon\delta\epsilon, \gamma\alpha -
\delta\epsilon, \gamma\beta, \delta\epsilon\delta,
\epsilon\gamma\beta\rangle\]
This algebra is triangular reduced, so the only possible reductions
are vertex removal or arrow removal.  All arrows occur in a minimal
set of relations, so this only leaves us with vertex removal
reduction. 

The injective and the projective dimensions of the simple
$\Lambda$-modules are given as follows. 
\begin{center}
\begin{tabular}{|c|c|c|c|}\hline
 & $S_1$ & $S_2$ & $S_3$ \\ \hline
 $\pd$ & $\infty$ & $\infty$ & 1\\ \hline 
 $\id$ & $\infty$ & 3 & $\infty$\\\hline
\end{tabular}
\end{center}
Using that the projective dimension of the simple module $S_3$ is $1$,
the vertex $3$ can be removed to obtain the algebra
\[\Lambda_1 = k\left(\vcenter{\xymatrix{ 1\ar[r]^{a_\alpha} \ar@/^1pc/[r]^{a_{\beta\epsilon}} &
      2\ar@/^/[l]^{a_{\gamma}}  }}\right)/\langle
a_\alpha a_\gamma, a_\gamma a_{\beta\epsilon}, a_{\beta\epsilon}a_\gamma\rangle.\]  
The indices $\sigma$ on $a_\sigma$ for the arrows in the quiver of $\Lambda_1$ and later $\Lambda_2$ refer to which basis elements in $\Lambda$ they correspond to.  This is the opposite
of the algebra in Example 3 in
\cite{Xi:FinDimConjII}, which originally appeared in
\cite{Igu-Sma-Tod:FinProjConFin}.  The algebra $\Lambda_1$ cannot be
reduced further.  

Alternatively, using that the injective dimension of the simple $\Lambda$-module
$S_2$ is $3$, the vertex $2$ can be removed to obtain the algebra
\[\Lambda_2 =
k\left(\vcenter{\xymatrix{1\ar@/^/[r]^{a_\beta} &
      3\ar@(ur,dr)^{a_{\epsilon\delta}}\ar@/^/[l]^{a_{\epsilon\gamma}}}}\right)/\langle a_\beta a_{\epsilon\gamma}, a_{\epsilon\delta}a_{\epsilon\delta}, a_{\epsilon\gamma}a_\beta, a_{\epsilon\delta}a_{\epsilon\gamma}\rangle.\]  
The algebra $\Lambda_2$ cannot be reduced further.  Hence, different
paths of reduction to a reduced algebra, do not give a unique algebra
up to isomorphism.

As $\Lambda_1$ and also $\Lambda_2$ are monomial algebras, it follows
easily from each of the different reductions that $\Lambda$ has finite
finitistic dimension.  One can show that no indecomposable projective
module can be a submodule of a finitely generated projective module
without being a direct summand.  Hence, the finitistic dimension of
$\Lambda_2$ is $0$, so that $\fpd\Lambda \leq 3$ by Corollary
\ref{cor:Solberg}.
\end{exam}

Not all examples in the existing literature are reducible, for
instance, the examples \cite[example at the end of Section 3]{Wang-Xi}
and \cite[Example 3.1 and 3.2]{Xu} seem to be reduced.

\appendix
\section{Gr\"obner basis}
In this appendix we recall some basic facts about Gr\"obner basis that
we use in Section \ref{sec:cleft-findim}.  For further details and a proper
introduction to the Gr\"obner basis we refer the reader to
\cite{Green}.  

Let $Q$ be a quiver and $kQ$ the corresponding path algebra over a
field $k$.  Then the set $\mathcal{B}$  of all the paths in $Q$ is a
$k$-basis of the path algebra $kQ$.  Gr\"obner basis theory is based
on having a totally ordered basis with special properties, namely an 
admissible ordering of the basis.  For a path algebra $kQ$ we obtain
such a basis by giving $\mathcal{B}$ the length-left-lexicographic
ordering $\succ$ by describing a total order on all vertices and a total order
on all arrows with all arrows bigger than any vertex.  Then, given any
non-zero element
\[r = \sum_{p\in \mathcal{B}} a_pp\] 
in $kQ$ with almost all $a_p$ in $k$ being zero and $p$ in
$\mathcal{B}$, we define the \textbf{tip of $r$} to be the element
$\Tip(r)=p$ if $a_p\neq 0$ and $p\succ q$ for all $q$ with
$a_q \neq 0$.  For any subset $X$ in $kQ$, then the \textbf{set of
  tips of $X$} is given as 
\[\Tip(X) = \{\Tip(r)\mid r \in X\setminus \{0\}\}\]
and the \textbf{set of non-tips of $X$} is given as 
\[\Nontip(X) = \mathcal{B}\setminus \Tip(X).\]
Having the notions of tips and non-tips of a set give rise to the
following fundamental result (see \cite[Lemma 5.1]{Green} for (i) and 
\cite[Theorem 2.1]{Green2} for (ii) and (iii)). 
\begin{lem}\label{lem:normalform}
 Let $I$ be an ideal in $kQ$ with $\succ$ an admissible ordering of
  the $k$-basis $\mathcal{B}$ consisting of all paths in $Q$.
\begin{enumerate}[\rm(i)]
\item For two paths $p$ and $p'$ in $Q$ and an element $x$ in $kQ$
  such that $pxp'$ is non-zero, the tip of $pxp'$ is 
\[\Tip(pxp') = p\Tip(x)p'.\]
\item As a $k$-vector space $kQ$ can be decomposed as
\[kQ = I \oplus \Span_k(\Nontip(I)),\]
where $\Span_k(\Nontip(I))$ denotes the $k$-linear span of
$\Nontip(I)$.  
\item As a $k$-vector space $kQ/I$ can be
indentified with $\Span_k(\Nontip(I))$.  In particular, any
elemenet $r + I$ in $kQ/I$ can be represent uniquely by $N(r)$ as
\[r  + I= N(r) + I,\]
where $N(r)$ is called the \textbf{normal form of $r$} and
  $N(r)\in\Span_k(\Nontip(I))$. 
\end{enumerate}
\end{lem}
A Gr\"obner basis of the ideal $I$ facilitates a way of
computing the normal form of any element in $kQ$.  A Gr\"obner basis
of an ideal in $kQ$ is defined as follows. 
\begin{defn}\label{defn:Groebner}
  Let $I$ be an ideal in $kQ$ with $\succ$ an admissible ordering of
  the $k$-basis $\mathcal{B}$ consisting of all paths in $Q$.  If
  $\mathcal{G}$ is a subset of $I$, then $\mathcal{G}$ is a Gr\"obner
  basis for $I$ with respect to $\succ$ if the ideal generated by
  $\Tip(\mathcal{G})$ equals the ideal generated by $\Tip(I)$.
\end{defn}
Having a finite Gr\"obner basis $\mathcal{G}$ for an ideal $I$ in $kQ$
with an admissible ordering $\succ$, gives a finite algorithm for
computing the normal form of any element in $kQ$ by iteratively
applying of the reduction described in statement (iii) of the result
below.  
\begin{lem}\label{lem:existencegroebner}
  Let $I$ be an admissible ideal in $kQ$ with $\succ$ an admissible
  ordering of the $k$-basis $\mathcal{B}$ consisting of all paths in
  $Q$.  
\begin{enumerate}[\rm(i)]
\item Then there exists a finite Gr\"obner basis $\mathcal{G}$ of $I$
  in $kQ$. 
\end{enumerate}
For a Gr\"obner basis $\mathcal{G}$ of $I$ in $kQ$ and any non-zero
  element $x$ in $I$, the following hold.
\begin{enumerate}[\rm(i)]
\setcounter{enumi}{1}
\item The element $\Tip(x)$ is in $\langle \Tip(\mathcal{G}) \rangle$.
\item There exist $c$ in $k$, paths $p$ and $p'$ in $Q$ and $g$ in $\mathcal{G}$
  such that  
\[\Tip(x) = p\Tip(g)p'\]
and
\[\Tip(x - cpgp') \prec \Tip(x),\]
whenever $x - cpgp' \neq 0$. 
\end{enumerate}
\end{lem}
For a proof of (i) see \cite[Corollary 2.2]{Green2}.  The statement
in (ii) is a consequence of the definition of a Gr\"obner basis, and
the statement in (iii) is a consequence of (ii) and the divsion
algorithm in \cite[Division algorithm 2.3.2]{Green2}. 

From a generating set $\mathcal{F}$ of an ideal in $kQ$ using the
Buchberger-algorithm a Gr\"obner basis $\mathcal{G}$ for $I$ can be
constructed (see \cite[2.4.1]{Green2}).  Using this algorithm it is
easy to see the following result, which we need in Section
\ref{sec:cleft-findim}.  Recall that an element $u$ in $kQ$ is called
\textbf{uniform} if $u = euf$ for some trivial paths $e$ and $f$ in
$kQ$.
\begin{lem}\label{lem:groebneravoidingarrow}
  Let $\mathcal{F}$ be a generating set of uniform elements of an
  ideal $I$ in $kQ$.  If an arrow $a$ does not occur in any path of
  any element of $\mathcal{F}$, then there is a Gr\"obner basis
  $\mathcal{G}$ of $I$ such that $a$ does not occur in any path of any
  element in $\mathcal{G}$.
\end{lem}
The second result we need in Section \ref{sec:cleft-findim} is an easy
consequence of the above.  
\begin{lem}\label{lem:nontipsarrowremoved}
Let $\Lambda = kQ/I$ for an admissible ideal $I$, and let $a$ be an
arrow in $Q$ not occurring in any path of any element of a minimal set
of generators for $I$.  Denote by $Q^*$ the quiver $Q$ with the arrow
$a$ removed, and let $I^* = kQ^*\cap I$.  Then
\[\Nontip(I^*) = \Nontip(I) \cap kQ^*.\]
\end{lem}
\begin{proof}
  Let $x$ be in $\Nontip(I^*)$ and assume that $x$ is in $\Tip(I)$.
  Then $x = p\Tip(g)p'$ for some paths $p$ and $p'$ and $g$ in
  $\mathcal{G}$. Since $x$ is in $kQ^*$, the paths $p$ and $p'$ are
  also in $kQ^*$.  By Lemma \ref{lem:groebneravoidingarrow} the
  Gr\"obner basis $\mathcal{G}$ for $I$ is also a Gr\"obner basis for
  $I^*$.  We infer that $x$ is a tip of an element in $I^*$, which is
  a contradiction, and $x$ is in $\Nontip(I)\cap kQ^*$.

  Conversely, assume that $x$ is in $\Nontip(I)\cap kQ^*$. Clearly $x$
  is the set $\mathcal{B}^*$ of all paths in $Q^*$ and not in
  $\Tip(I)$.  Since $I^*\subseteq I$, the inclusion
  $\Tip(I^*) \subseteq \Tip(I)$ holds.  If $x$ is in $\Tip(I^*)$, then
  $x$ is in $\Tip(I)$. This is a contradication, so that $x$ is in
  $\Nontip(I^*)$. 
\end{proof}

\end{document}